\documentclass[10pt]{article}

   \usepackage{amssymb,amsmath,amsthm}
   \usepackage[dvipsnames]{xcolor}
   \usepackage{mathtools}
   \usepackage{dsfont}
   \usepackage{graphicx}
   \usepackage{natbib}

   \mathtoolsset{showonlyrefs}

   \newtheorem{Theorem}{Theorem}
   \newtheorem{Proposition}{Proposition}
   \newtheorem{Lemma}{Lemma}
   \newtheorem{Corollary}{Corollary}

   \newtheorem{Remark}{Remark}

   \makeatletter
   \def\calFactory#1{%
   \expandafter\def\csname c#1\endcsname{\mathcal{#1}}}
   \def\frakFactory#1{%
   \expandafter\def\csname k#1\endcsname{\mathfrak{#1}}}
   \def\bbFactory#1{%
   \expandafter\def\csname b#1\endcsname{\mathbb{#1}}}
   \def\bfFactory#1{%
   \expandafter\def\csname f#1\endcsname{\mathbf{#1}}}
   \newcounter{ctr}
   \loop
   \stepcounter{ctr}
   \edef\X{\@Alph\c@ctr}%
   \edef\x{\@alph\c@ctr}%
   \expandafter\calFactory\X
   \expandafter\frakFactory\X
   \expandafter\frakFactory\x
   \expandafter\bbFactory\X
   \expandafter\bfFactory\X
   \ifnum\thectr<26
   \repeat
   \makeatother

   \newcommand{\mathset}[1]{\mathbb{#1}}

   \newcommand{\N}{\mathset{N}}

   \newcommand{\R}{\mathset{R}}
   \newcommand{\ER}{\mathset{R}}

   \newcommand{\F}{\mathcal{F}}

   \newcommand{\mathproba}[1]{\mathbf{#1}}

   \newcommand{\e}{\mathproba{E}}
   \newcommand{\p}{\mathproba{P}}

   \newcommand{\HUN}{\mathbf{(H_1)}}
   \newcommand{\HDEUX}{\mathbf{(H_2)}}
   \newcommand{\HTROIS}{\mathbf{(H_3)}}

   \newcommand{\Lip}{\mathrm{Lip}}
   \newcommand{\xx}{{x_0}}
   \newcommand{\betaH}{{\beta_H}}
   \newcommand{\ah}{{a_H}}
   \renewcommand{\x}{\xx}

   \newcommand{\kkk}{\kK}
   \newcommand{\constPV}{\kC}
   \newcommand{\LK}{\kL_K}

   \DeclareMathOperator*{\argmin}{arg\,min}

  \author{Karine Bertin\footnote{CIMFAV-INGEMAT, Universidad de Valparaiso, General Cruz 222, Valpara\'iso, Chile.
   E-mail: \texttt{karine.bertin@uv.cl}},
       Nicolas Klutchnikoff\footnote{Universit\'e de Rennes, CNRS, IRMAR~--~UMR 6625, F-35000 Rennes, France,
   E-mail: \texttt{nicolas.klutchnikoff@univ-rennes2.fr}},
        Fabien Panloup\footnote{Laboratoire  angevin  de  recherche  en  math\'ematiques  (LAREMA),  Universit\'e d'Angers, CNRS, 49045 Angers Cedex 01, France, E-mail: \texttt{fabien.panloup@univ-angers.fr}},
         Maylis Varvenne\footnote{ Institut de Math\'ematiques de Toulouse (IMT), Universit\'e de Toulouse 1 Capitole, 2 Rue du Doyen-Gabriel-Marty, 31042 Toulouse, France.  E-mail: \texttt{maylis.varvenne@math.univ-toulouse.fr; ~maylis.varvenne@ut-capitole.fr}}
}

 \title{Adaptive estimation of the stationary density of a stochastic differential equation driven by a fractional Brownian motion
   }

\begin{document}
\maketitle

\begin{abstract}
We build and study a data-driven procedure for the estimation of the stationary density $f$ of an additive fractional SDE. To this end, we also prove some new concentrations bounds for discrete observations of such dynamics in stationary regime.\end{abstract}

\medskip
\noindent \textit{AMS classification (2010)}: 60G22, 60H10, 62M09.





%
%


   \maketitle


   \section{Introduction}\label{sec:introduction}
   We consider the $\R^d$-random process $X=(X_t : t\geq 0)$ governed by stochastic differential equation
   \begin{equation}\label{modeldiffusion}
      X_t = X_0 + \int_0^t b(X_s) ds + \sigma B^H_t,
      \quad t\geq 0,
   \end{equation}
   where $X_0$ is the initial value of $X$, $b\colon \R^d\to \R^d$ is a continuous function, $\sigma$ is a constant $d\times d$ matrix and $B^H = (B_t^H\colon t\geq 0)$ is a $d$-dimensional two-sided fractional Brownian motion with Hurst parameter $0<H<1$.
   Even in this non-markovian framework (if $H\neq 1/2$), the process $X$ can be embedded into an infinite Markovian structure \citep[see][]{hairer}. This allows us to define, under usual assumptions on the coefficients $b$ and $\sigma$, a \emph{unique} invariant distribution of  $X$ which admits a density $f:\R^d\to\R$.
   In this paper we are interested in the non-parametric  estimation of $f$  based on the observation of $X$ at $n$ equally spaced sampling times $t_1 = \Delta_n, \dots, t_n=n\Delta_n$ {where $\Delta_n$ is a non-increasing positive sequence such that $n\Delta_n\to\infty$.}

   In the case of diffusion models driven by standard Brownian motion ($H=\frac{1}{2}$), the problem of non parametric estimation of the invariant density  has been extensively studied, in both discrete and continuous time. In the continuous time framework, the process $X$ is observed for all $0\leq t \leq T$. \Citet{castellana1986smoothed} proved that, under some specific assumption on the joint density of $(X_0, X_t)$, the parametric rate of convergence $T^{-1/2}$ can be reached.  Among other, see also \citet{bosq1997parametric}, \citet{kutoyants1998efficient}, \citet{dalalyan2001estimation}, \citet{comte2005super} and \citet{bosq2012nonparametric}. Without this specific assumption, classical non-parametric rates of convergence of the form $T^{-s/(2s+1)}$ can be obtained \citep[see][]{comte2002adaptive} where $s$ is the smoothness parameter of the function $f:\R\to \R$.
   The case of discrete observations (which corresponds to our framework) has been studied in a univariate setting in \citet{tribouley1998l_p},  \citet{comte2002adaptive}  and \citet{schmisser2013nonparametric} for integrated risk and in \citet{MR3638047} for pointwise risk.  In these papers, the rate of convergence is proved to depend only on $T = n\Delta_n$ and (adaptive) minimax rates of convergence are of the form $(n\Delta_n)^{-s/(2s+1)}$ (up a a logarithmic term) where $s$ is the smoothness of the density function. See also \citet{bertin2018adaptive} that consider integrated risk in a multivariate setting.
   When $H\neq 1/2$, nonparametric estimation methods for the model~\eqref{modeldiffusion} have mainly  focused on estimation of the drift term $b$ on the continuous case, see  e.g. \citet{mishra2011nonparametric} (where the authors study the consistency and the rate of convergence of a nonparametric estimator of the whole trend of the solution to a fractional SDE) and \citet{comte2018nonparametric} (where the authors the consistency of some Nadaraya-Watson's-type estimators of the drift function in a fractional SDE).  Note that these papers only consider the case $H>1/2$ in the continuous case.

   \noindent Our goal in this paper is to construct a data-driven procedure to estimate the stationary density $f$ of $X$ in the discrete case for both $H<1/2$ and $H>1/2$. To this aim, new concentration inequalities are obtained for the ``stationary'' process, following the strategy of \citet{varvenne2019concentration}. In this paper, the idea was to use a pathwise interpretation of  the concentration phenomenon by studying the distance between a functional and its average as a sum of differences of ``conditioned paths''. Then, the result was obtained by making use of the  contraction properties of the dynamics (under  strong convexity assumptions). In our paper, the novelty with respect to this paper is to assume that $\Delta_n$ may depend on $n$ but mostly, that the process is observed in its stationary regime (instead of starting from a given $x$ like in \cite{varvenne2019concentration}). Actually, if this modification is easy to overcome in a Markovian setting, here, this is not the case since at time $0$, the process has already a past. In other words, an invariant distribution of \eqref{modeldiffusion} is  a probability on $\ER^d\times{\cal W}$ where ${\cal W}$ is a functional space (see Section \ref{sec:modell} for details). In short, proving concentration bounds in stationary regime requires to strongly modify the original proof given in \cite{varvenne2019concentration} (see Section \ref{subsec:sketch} for more detailed explanations).

   \noindent These tools are used for two purposes. First we obtain rates of convergence for the pointwise risk of classical kernel estimators assuming that $f$ belongs to a H\"older class with a known smoothness parameter $s=(s_1,\dots,s_d)\in(0,\infty)^d$. More precisely, choosing adequately a bandwidth that depends on $s$, we obtain the rate $\phi_n(s)=(n\Delta_n)^{-\betaH\gamma(\mathbf{s})}$ where
   \begin{equation}\label{eq:def-gamma-beta}
      \gamma(\mathbf{s}) = \frac{\bar s}{2\left(1+\frac1{\min_j s_j}\right)\bar s +2}
      \qquad\text{and}\qquad
      \betaH = 2 - \max(2H, 1).
   \end{equation}
   Here $\bar s=\left( \sum_{i=1}^d 1/s_i\right)^{-1}$ denotes a classical parameter in multivariate nonparametric estimation that can be viewed as the effective smoothness of $f$.
   Next, we propose a data-driven procedure based on the ideas developed by Goldenshluger and Lepski \citep[see][and references therein]{goldenshluger2011bandwidth,goldenshluger2014adaptive} to select the bandwidth. The concentration tools we develop in this paper allow us to prove an oracle-type inequality. This ensures that our data-driven procedure performs almost as well as the best estimator in a given family of estimators. As a direct consequence, our procedure is proved to be adaptive: assuming that $f$ is H\"older with unknown smoothness $s$,  it converges at the rate $\phi_n(s)$ up to a $\log (n\Delta_n)$ factor.

   {The paper is organized as follows. We first present the model and the new concentration inequalities in Section~\ref{sec:mod}. We introduce the statistical framework in Section~\ref{sec:adapt}. Section~\ref{sec:estim} is devoted to the description of our estimation procedures and their theoretical properties are stated in Section~\ref{sec:results}. The proofs are postponed to Section~\ref{sec:proofs1} (for the concentration inequalities) and~\ref{sec:proofs2} (for the properties of statistical procedures).}

   \section{Model and Probabilistic background}\label{sec:mod}

   \subsection{Model}\label{sec:modell}

   We recall that  in the non-Markovian setting given by~\eqref{modeldiffusion}, the well definition of ``the'' invariant distribution of the process $X$ requires the embedding of the dynamics into an infinite-dimensional Markovian structure. More precisely,  the Markovian process above the dynamics, called \textit{Stochastic Dynamical System} (SDS) can be realized as a map on the space
   $\mathbb{R}^d\times {\mathcal W}$ where ${\mathcal W}$ denotes an appropriate space of H\"older functions from $(-\infty,0]$ to $\mathbb{R}^d$, equipped with the Wiener measure. This construction is strongly based on the Mandelbrot Van-Ness representation of the fBm:
   \begin{equation}\label{eq:mandelbrot}
   \forall t\in\R,\quad B_t^H=c_H\int_{\R}(t-s)_+^{H-1/2}-(-s)_+^{H-1/2} {\rm d}W_s,
   \end{equation}
   where $(W_t)_{t\in\R}$ is a two-sided $d$-dimensional Brownian motion and $c_H>0$. We denote {by $({\mathcal Q}_t(x,w))_{t\ge0,(x,w)\in\ER^d\times{\mathcal W}}$}  the related semi-group (for details on regularity properties of the SDS, see \cite{hairer}).

   \noindent {For this type of dynamics, an \textit{initial condition} is given by a couple $(X_0,W^-)$, where $W^-=(W_t)_{t\le0}$ and $X_0\in\ER^d$. In other words, an \textit{initial condition} is a distribution $\mu$ on $\ER^d\times {\mathcal W}$ such that the projection on the second coordinate is $\mathproba{P}_{W^-}$.}

   \noindent {Then, an invariant distribution $\nu$ for {$({\mathcal Q}_t(x,w))_{t\ge0,(x,w)\in\ER^d\times{\mathcal W}}$} is an initial condition which is such that the distribution $\mathproba{P}_{X^\nu}$ of the process $(X_t^\nu)_{t\ge0}$ built with this initial condition is invariant by a time-shift. We say that the invariant distribution is unique if $\mathproba{P}_{X^\nu}$ is unique. Finally, if the invariant distribution exists, we will denote by $\bar{\nu}$,  its first marginal: $\bar{\nu}(dx)=\int_{\mathcal W}\nu(dx,dw)$. Such a distribution (on $\ER^d$) will be usually called ``marginal invariant distribution''.
   We will denote by $f$ the density of $\bar{\nu}$ with respect to the Lebesgue measure on $\ER^d$ (denoted by $\lambda_d$ in the sequel) {when exists}.
   }
   In Proposition \ref{pro:smoothinvariant} below, we recall some {sufficient} conditions which ensure existence, uniqueness of the invariant distribution and absolute continuity of $\bar{\nu}$ with respect to the Lebesgue measure. To this end, let us first state the assumptions used throughout our paper:

   \noindent $\HUN$ \textit{(stability)}
   The function $b:\R^d\rightarrow\R$ is continuous and
   there exists a constant $\alpha>0$ such that:
   \noindent
   For every $x,y\in\R^{d}$, we have
   \begin{equation*}
      \langle  b(x)-b(y), \, x-y\rangle \le - \alpha |x-y|^2
   \end{equation*}

   \noindent $\HDEUX$ \textit{(strong regularity)} For every $x,y\in\R^{d}$,
   \begin{equation*}
   |b(x)-b(y)|\leq L|x-y|.
   \end{equation*}

   \noindent $\HTROIS$ \textit{(nondegeneracy)} The matrix $\sigma$ is invertible.

   \begin{Proposition}\label{pro:smoothinvariant}
      {Assume $\HUN$ and $\HDEUX$. Then, existence holds for the invariant distribution $\nu$. If $\HTROIS$ is also fulfilled, then uniqueness holds for $\nu$ \citep[unique in the sense of][]{hairer}.} Furthermore,  {if $b$ is ${\cal C}^1$}, then the marginal invariant distribution $\bar{\nu}$ admits a density $f$ with respect to $\lambda_d$.
   \end{Proposition}

   Existence and uniqueness are consequences of \cite{hairer}. For the existence of density for $\bar{\nu}$,
   we rely on the one hand, on \cite[Theorems 1.2, 1.3]{hairer}, which state that  $(X_t^x)$ converges in total variation distance towards $\bar{\nu}$ and, on the other hand, to the fact that, under $\HTROIS$, the distribution of $(X_t^x)$ has a density with respect to the Lebesgue measure for any $t>0$ (see $e.g.$ \cite[Theorem 1.2]{Besalu_kohatsu_tindel} or \cite[Theorem 4.3]{baudoin_hairer} when $H>1/2$). The combination of these two properties implies that $\bar{\nu}$ is absolutely continuous $w.r.t.$ to the Lebesgue measure (or equivalently that the density exists).\smallskip


   \begin{Remark} $\rhd$ {For the existence of the invariant distribution}, Assumption $\HUN$ could be alleviated. More precisely, the contraction assumption may be only assumed out of a compact set. However, we chose to recall the result only under this assumption since, in the sequel, we will need $\HUN$ to obtain concentration properties.\\

   \noindent $\rhd$ {In this paper, we do not discuss about the smoothness of $f$.  {This problem} is out of the scope of the paper. However, we can expect that the smoothness of $f$ strongly relies on the one of $b$. For instance, in the setting of gradient diffusions $dx_t=-\nabla U(x_t)+\sigma_0 dW_t$ (where $\sigma_0$ is a positive number) , it is well-known that the density is given by $ f(x)=\frac{1}{Z_{U,\sigma_0}} \exp\left(-\frac{U(x)}{2\sigma_0^2}\right)$.   This involves that in this particular case, when $b$ is of class ${\cal C}^{r+\alpha}$ ($\alpha\in(0,1]$), then $f$ is of class ${\cal C}^{r+1+\alpha}$. We conjecture that this property is still true in our setting.}
   \end{Remark}

   \subsection{Concentration inequalities for stationary solution}

   Let $n\in\N^*$. We denote by $d_n$ the following $L^1$--distance:
   \begin{equation}\label{eq:def_dn}
   \forall (x,y)\in(\R^d)^n\times(\R^d)^n ,\quad d_n(x,y):=\sum_{k=1}^n|x_i-y_i|,
   \end{equation}
   where $|\,.\,|$ stands for the Euclidean norm on $\ER^d$. For a given $d\times d$-matrix $A$ with real entries, we also denote by $\|\,.\,\|$ a given matrix-norm, subordinated to the Euclidean norm.

   \begin{Theorem}\label{main_thm_concentration}
   Let $H\in(0,1)$. Assume {$\HUN$, $\HDEUX$ and $\HTROIS$ and denote by $(X_t)_{t\geq0}$ the stationary solution associated to \eqref{modeldiffusion}}. Let $n\in\N^*$ and $\Delta_n>0$ {such that $n\Delta_n\geq~1$}.
   Let $d_n$ be the metric defined by \eqref{eq:def_dn}. Then, there exists some positive constant {$\constPV=C(H,L,\alpha,|b(0_{\R^d})|,\|\sigma\|)$} such that for all Lipschitz function $F:\left((\R^d)^n,d_n\right)\to(\R,|\cdot|)$ and for all $\lambda>0$,
   \begin{equation}\label{eq:moment_exp_F}
   \e\left[\exp\left(\lambda(F_X-\e[F_X])\right)\right]\leqslant \exp\left(\constPV\|F\|^2_{\rm Lip}\lambda^2n^{\ah}\Delta_n^{-\betaH}\right).
   \end{equation}
   where $\ah:=\max\{2H,1\}$, $\betaH:=\min\{1,2-2H\}$ {and $F_X:=F(X_{t_1},\dots,X_{t_n})$}.\\
   Moreover, we deduce from the previous inequality that
   \begin{equation}\label{eq:concentration_F}
   \p\left(F_X-\e[F_X]>r\right)\leqslant \exp\left(\frac{-r^2}{4\constPV n^{\ah}\Delta_n^{-\betaH}\|F\|^2_{\rm Lip}}\right).
   \end{equation}

   \end{Theorem}

   \begin{Corollary} \label{cor:concentration_inequality}
   Let the assumptions of Theorem \ref{main_thm_concentration} be in force. Let $F_X:=\frac{1}{n}\sum_{k=1}^ng(X_{t_k})$ where $g:(\R^d,|\cdot|)\to(\R,|\cdot|)$ is a given Lipschitz function. We have $\|F\|^2_{\rm Lip}=n^{-2}\|g\|_{\rm Lip}^2$ and then there exists some positive constant {$C=C(H,L,\alpha,|b(0_{\R^d})|,\|\sigma\|)$} such that
   \begin{equation}\label{eq:concentration_discrete_sum}
   \p\left(\frac{1}{n}\sum_{k=1}^n\left(g(X_{t_k})-\e[g(X_{t_k})]\right)>r\right)\leqslant \exp\left(\frac{-r^2n^{\betaH}\Delta_n^{\betaH}}{4\constPV \|g\|^2_{\rm Lip}}\right)
   \end{equation}
   since $2-\ah=\betaH$.
   \end{Corollary}

   \begin{Remark}\label{rem:theoconcen}{$\rhd$ In Theorem \ref{main_thm_concentration}, assumption $\HTROIS$ ensures the uniqueness of the stationary solution $(X_t)_{t\geq0}$ but the concentration result remains true for every stationary solution to \eqref{modeldiffusion} when $\HTROIS$ does not hold. } \\
   \noindent $\rhd$ In the above results, the constant $\constPV=C(H,L,\alpha,|b(0_{\R^d})|,\|\sigma\|)$ can be chosen in such a way that  $(L,\alpha,b_0,q)\mapsto C(H,L,\alpha,b_0,q)$ is bounded on every compact set of $\ER_+\times\ER_+^*\times\ER_+\times\ER_+$ (see Remarks \ref{rem:constant1} and \ref{rem:constant2} for more details).\\
   \noindent $\rhd$ In Corollary \ref{cor:concentration_inequality}, we remark that concentration bounds can be deduced from \eqref{eq:concentration_discrete_sum} if we impose at least that $\lim\limits_{n\to+\infty}n\Delta_n=+\infty$, $i.e.$ that $\lim\limits_{n\to+\infty} t_n=+\infty$.
   \end{Remark}

   \section{Adaptive framework}\label{sec:adapt}

   {Let us recall that $\bar{\nu}$ denotes the marginal invariant distribution and that $f$ denotes its density $w.r.t.$ the Lebesgue measure $\lambda_d$, which is assumed to exist in whole the paper (see Proposition~\ref{pro:smoothinvariant} for conditions of existence).}
   To measure the accuracy of an estimator $\tilde f = \tilde f(\cdot, X_{t_1},\dotsc,X_{t_n})$ of $f$, we define the pointwise risk
   \[
   R_n(\tilde f, f) = \left(\e\left|\tilde f(\xx)-f(\xx)\right|^p\right)^{1/p}
   \]
   where $p\geq 1$ is fixed.
   Let $\bF$ be a subset of $\cC(\R^d,\R)$. In what follows, we will consider specific H\"older classes. The maximal risk of $\tilde{f}_n$ over $\bF$ is defined by:
   \begin{equation}
      R_n(\tilde{f},\bF) = \sup_{f\in\bF} R_n(\tilde{f},f).
   \end{equation}
   We say that an estimator $\tilde{f}$ converges at the rate of convergence $\phi_n(\bF)$ over $\bF$  if
   \begin{equation}\label{eq:rateXX}
   \limsup_{n\to \infty} \phi_n^{-1}(\bF)  R_n(\tilde{f},\bF) <\infty
   \end{equation}
   Note that such estimator may depend on the class $\bF$. Moreover \eqref{eq:rateXX} ensures a specific behavior of  the estimator $\tilde{f}$ over $\bF$ but the same estimator can perform poorly over another functional space. The problem of adaptive estimation consists in finding a single estimation procedure with a good behavior over a scale of functional classes. More precisely, given a family $\{\bF_\lambda : \lambda\in\Lambda\}$ of subsets of $\cC(\R^d,\R)$, the goal is to construct $f^*_n$ such that $R_n(f^*_n, \bF_\lambda)$ is asymptotically bounded, up to a small multiplicative factor (for example a constant or a logarithmic term), by $\phi_n(\bF_\lambda)$ for any $\lambda\in\Lambda$. One of the main tools to prove that an estimation procedure is adaptive over a scale of functional classes is to prove an oracle-type inequality that guarantees that this procedure performs almost as well as the best estimator in a rich family of estimators. Ideally, we would like to have an inequality of the following form:
   \begin{equation}\label{eq:ideal-oracle}
      R_n(f^*, f) \leq \inf_{\eta\in H} R_n(\hat f_\eta, f),
   \end{equation}
   where $\{\hat f_\eta\colon \eta\in H\}$ is a family of estimators satisfying: for any $\lambda\in\Lambda$, there exists $\eta(\lambda)$ such that $\hat f_{\eta(\lambda)}$ converges at the rate $\phi_n(\bF_\lambda)$ over the class $\bF_\lambda$.
   {In general, obtaining such an inequality is not possible. However in many situations, \eqref{eq:ideal-oracle} can be relaxed and a weaker inequality of the following type can be proved:}
   \begin{equation}\label{eq:oracle}
      R_n(f^*, f) \leq \Upsilon_1\inf_{\eta\in H} R_n^*(f, \eta) + \Upsilon_2\delta(n),
   \end{equation}
   where $\Upsilon_1$ and $\Upsilon_2$ are two positive constants, $R_n^*(f, \eta)$ is an appropriate quantity to be determined that can be viewed as a tight upper bound on $R_n(\hat f_{\eta}, f)$ and $\delta(n)$ is a reminder term. Inequalities of the form~\eqref{eq:oracle} are called \emph{oracle-type} inequalities.

   \section{Estimation procedure}\label{sec:estim}

   {To estimate $f$ we construct a procedure defined through classical kernel density estimators. It is well known that the accuracy of these estimators is mainly determined by  the \emph{bandwidth} vector. Thus, obtaining a \emph {data-driven} choice of this parameter is the central problem in our model. In this section, after introducing a family of kernel density estimators, we define a selection procedure based on the ideas developed in~\citet{GL2011}.}

   \subsection{Kernel density estimators}\label{sec:kernel}

   In this paper a function $K:\R\to\R$ is called a kernel if the support of $K$ is included into $[-1,1]$, $K$ is a Lipschitz function with Lipschitz constant $\LK>0$ and $K$ satisfies $\int_{\R} K(u) \mathrm{d}u= 1$. Following \cite{rosenblatt1956remarks} and \cite{parzen1962estimation}, we consider kernel density estimators $\hat f_h$ defined, for $h = (h_1,\dotsc,h_d)\in(0,+\infty)^d$, by:
   \begin{equation}\label{eq:def-kernel-estim}
      \hat f_h(\xx) = \frac1n\sum_{i=1}^n K_h(\xx-X_{t_i}),
   \end{equation}
   where, for $(u_1,\dotsc,u_d)\in\R^d$, we define $K_h(u_1,\dotsc, u_d)=\prod_{i=1}^d h_i^{-1}K(h_i^{-1}u_i)$.
   We say that a kernel $K$ is of order $M\in\N$ if for any $1\leq \ell\leq M$, we have $\int_{\mathbb{R}} K(x) x^\ell dx = 0$. In the following paragraph, we propose a data-driven procedure to select the bandwidth $h$ in the finite subset $\cH$ of $(0,1)^d$.

   \subsection{Bandwidth selection}\label{sec:bandwidth-selection}

   Our procedure depends on a hyperparameter $\kkk>0$. We refer the reader to Section~\ref{sec:results} for  detailed comments on the impact of the choice of this parameter.
   For any bandwidth vector $h=(h_1,\dots,h_d)\in(0,1)^d$, we define:
   \begin{equation}\label{eq:Vhvarphi}
      V_h = \min(h_1,\dotsc,h_d)\prod_{i=1}^d h_i
      \quad\text{and}\quad
      \varphi_n(h) =
      \left(\frac{4d\LK^2}{V_h^2(n\Delta_n)^{\betaH}}\right)^{1/2},
   \end{equation}
   where $\betaH$ is defined by~\eqref{eq:def-gamma-beta}.
   We consider the following subset of $(0,1)^d$
   \begin{equation*}
      \cH = \left\{
      (e^{-l_1}, \dotsc, e^{-l_d}) : l_i = 0,\dotsc, \left\lfloor\frac{\betaH}{2}
      \log\left(n\Delta_n\right)\right\rfloor
      \right\} \cap \widetilde\cH
   \end{equation*}
   where $\lfloor\cdot\rfloor$ denotes the integer part and
   \begin{equation*}
      \widetilde{\cH} = \left\{
      h \in (0, 1)^d :
      V_h \geq \left( \frac{1}{n\Delta_n} \right)^{\betaH/2}
      \right\}.
   \end{equation*}
   {Without loss of generality, we assume that $(n\Delta_n)^{\beta_H}\ge e$, which implies in particular that $\cH$ is not empty.}

   Following \cite{GL2011}, we define for $h$ and $\kh$ in $\cH$ the following quantities
   \begin{equation}\label{eq:MNHH}
      M_n(h,\kh) =  M_n(\kh)+ M_n(\kh\vee h)
      \quad\text{with}\quad
      M_n(h) = \varphi_n(h)\sqrt{\kkk p |\log V_h|}
   \end{equation}
   and
   \begin{equation}
      B(h,\xx) = \max_{\kh\in\cH}\left\{\left|\hat f_{\kh\vee h}(\xx)-\hat f_{\kh}(\xx)\right|-{M}_n(h,\kh)\right\}_+.
   \end{equation}
   Here $\{y\}_+=\max(0,y)$ denotes the nonnegative part of $y\in\R$ and  $\kh\vee h$ denotes the component-wise maximum of the bandwidth $h$ and $\kh$.
   Our procedure consists of selecting a bandwidth $\hat h(\x)$ such that
   \begin{equation}\label{eq:selection-rule}
      \hat h(\xx) = \argmin_{h\in\cH} \left(B(h,\xx)+ M_n(h)\right).
   \end{equation}
   The final estimator of $f(\xx)$ is then defined as the plugin estimator:
   \begin{equation}
      \hat f(\xx) = \hat f_{\hat h(\xx)}(\xx).
   \end{equation}

   This selection rule follows the principles and the ideas developed by Goldenshluger and Lepski. The quantity $M_n(h)$, called a \emph{majorant}, is a penalized version of the standard deviation of the estimator $\hat{f}_h$ while the quantity $B(h,\xx)$ is, in some sense, closed to its bias term. Finding tight \emph{majorants} is the key point of the method since $\hat{h}(x_0)$ is chosen in \eqref{eq:selection-rule} in order to realize an empirical trade-off between these two quantities.

   \section{Results}\label{sec:results}
   We first recall the definition of H\"older balls $\Sigma_d(\boldsymbol{s},\bold L)$. For two $d$-tuples of positive reals $\boldsymbol{s}=(s_1,\dots,s_d)$ and $\bold L=(L_1,\dots,L_d)$,
   \begin{eqnarray*}
   &&\Sigma_d(\boldsymbol{s},\bold L)=\Big\{ f: \R^d\to \R \text{  s.t. } \forall\, 1\leq i \leq d\quad
   \left\|\frac{\partial^{m}f}{\partial x_i^{m}}\right\|_\infty\leq L_i, \quad m=0,\dots,\lfloor s_i\rfloor\\
   && \text{ and for all } \,t\in \R \quad \left\|\frac{\partial^{\lfloor s_i\rfloor}f}{\partial x_i^{\lfloor s_i\rfloor}}(\cdot+te_i) - \frac{\partial^{\lfloor s_i\rfloor}f}{\partial x_i^{\lfloor s_i\rfloor}}(\cdot)\right\|_{\infty}\leq L_i |t|^{s_i- \lfloor s_i\rfloor}
   \Big\}\\
   \end{eqnarray*}
   where for any $i$, $\lfloor s_i\rfloor=\max\{l\in\mathbb{N}: l<s_i\}$ and $e_i$ is the vector where all coordinates are null except the $i$-th one which is equal to 1.

   \subsection{Properties of the kernel estimators}

   The two following propositions give upper-bounds of the bias and the stochastic term of the estimator $\hat{f}_h$ .

   \begin{Proposition}\label{prop:bias}
   Let $\mathbf{s}=(s_1,\dotsc, s_d)\in(0,+\infty)^d$ and $\mathbf{L}=(L_1,\dotsc, L_d)\in(0,+\infty)^d$. Assume that  $f\in \Sigma_d(\boldsymbol{s},\bold L)$ and assume that $K$ is a kernel of order greater than $\max_i \lfloor s_i \rfloor$. Under $\HUN$,  $\HDEUX$ and $\HTROIS$, we have for all $h\in\cH$
   \begin{equation}\label{eq:bias}
      \left|\e\hat f_h(\xx) -  f(\xx)\right|\le \sum_{i=1}^d \frac{L_i h_i^{s_i}}{\lfloor s_i \rfloor !} \int_\R \big|v^{s_i} K(v)\big| dv
   \end{equation}
   and
   \begin{equation}\label{eq:bias2}
   E_h(\x)
   = \max_{\kh\in\cH}
   \left|\e \hat{f}_{h\vee\kh}(x_0)-\e\hat{f}_{\kh}(x_0)\right|
   \le 2\sum_{i=1}^d \frac{L_i h_i^{s_i}}{\lfloor s_i \rfloor !} \int_\R \big|v^{s_i} K(v)\big| dv.
   \end{equation}
   \end{Proposition}

   \begin{Proposition}\label{prop:stoc}
   Under $\HUN$,  $\HDEUX$ and $\HTROIS$, we have for all $h\in\cH$
   \begin{equation}\label{eq:stoc}
      \left(\e\left|\hat f_h(\xx) - \e\hat f_h(\xx)\right|^p\right)^{1/p}
      \leq  \left(p\Gamma\left(\frac{p+1}{2}\right)\right)^{1/p}
      {\constPV^{1/2}}\varphi_n(h).
   \end{equation}
   \end{Proposition}
   \begin{Remark}
      $\rhd$ Note that the control of the bias term obtained in Proposition~\ref{prop:bias} is the same as those obtained for the problem of density estimation in an i.i.d. context. The control of the stochastic term, see Proposition~\ref{prop:stoc}, relies on the concentration inequality obtained in {Corollary~\ref{cor:concentration_inequality}}. {The right hand side of~\eqref{eq:stoc} depends on the additional assumptions made on our model through the constant $\constPV$.}\\
      \noindent $\rhd$ This result is valid for a large class of functional (only a Lipschitz condition is required) and under weak assumptions on the process. For $H=1/2$, the concentration inequality is optimal, see \citet{saussereau2012transportation} and~\citet{MR2078555}.  However, under strongest assumptions---for example on the joint distribution of $(X_0, X_t)$, Bernstein-type inequalities can be used to derive a better upper bound on the stochastic term of order $(nh_1\dotsc h_d)^{-1}$, see \citet{MR3638047} for $d=1$.
   \end{Remark}

   Using the above propositions we derive, over any H\"older balls $\Sigma_d(\boldsymbol{s},\bold L)$, the rate of convergence achieved by a kernel estimator defined in~\eqref{eq:def-kernel-estim} with a specific choice of bandwidth that depend on the smoothness parameter~$\boldsymbol{s}$.

   \begin{Theorem}\label{theo:rates}
      Let $\boldsymbol{s}=(s_1,\dotsc, s_d)\in(0,+\infty)^d$ and $\mathbf{L}=(L_1,\dotsc, L_d)\in(0,+\infty)^d$. Assume that  $f\in \Sigma_d(\boldsymbol{s},\bold L)$ and assume that $K$ is a kernel of order greater than $\max_i \lfloor s_i \rfloor$. Under $\HUN$,  $\HDEUX$ and $\HTROIS$, the estimator $\hat{f}_{h(\boldsymbol{s})}$ defined through the bandwidth $h(\boldsymbol{s})=\left(h_1(\boldsymbol{s}),\ldots,h_d(\boldsymbol{s})\right)$ where for any $i=1,\dotsc, d$
      \begin{equation*}
         h_i(\boldsymbol{s}) =  \left(\frac{1}{(n\Delta_n)^{\betaH}}\right)^{\frac{\gamma(\boldsymbol{s})}{s_i}}
         \quad\text{with}\quad
         \gamma(\boldsymbol{s}) = \frac{\bar s}{2\left(1+\frac1{\min_j s_j}\right)\bar s +2}\in(0, 1/2)
      \end{equation*}
      and satisfies
      \begin{equation}\label{eq:rate}
         R_n(\hat{f}_{h(\boldsymbol{s})}, f)\le    (\Lambda_1+\Lambda_2)
         \left(n\Delta_n\right)^{-{\betaH}\gamma(\boldsymbol{s})},
      \end{equation}
      where
      \begin{equation*}
      \Lambda_1 =
      \sum_{i=1}^d \frac{L_i}{\lfloor s_i \rfloor !} \int_\R \big|v^{s_i} K(v)\big| dv,
      \quad
      \Lambda_2 = 2\sqrt{d\constPV}\left(p\Gamma\left( \frac{p+1}2 \right)\right)^{1/p}\LK.
   \end{equation*}
   \end{Theorem}

   To our best knowledge few papers deal with nonparametric rate of convergence in our model. Only \citet{comte2018nonparametric} have considered the estimation of the trend function $b$ based on continuous observations when $d=1$ and $H>1/2$. They obtain the same rates of convergence only in the case $s=1$ assuming a Lipschitz condition on the function $b$.

   \subsection{Properties of the data-driven procedure}

   The estimator $\hat{f}$, defined in Section~\ref{sec:bandwidth-selection} using the hyperparameter $\kkk$ and the family of bandwidths $\cH$ satisfies the following oracle inequality.

   \begin{Theorem}\label{thm:oracle}
      Under $\HUN$, $\HDEUX$ and $\HTROIS$, if $\kkk>\constPV$, we have:
      \begin{equation}\label{eq:thm-oracle}
         R_n(\hat{f}, f) \leq
         \min_{h \in\cH} \left\{
            R_n(\hat{f}_h, f) +
            4 M_n(h)
            + 3 E_h(\x)
         \right\}
         + C_0(n\Delta_n)^{-\betaH/2}
      \end{equation}
      with
      \begin{equation*}
         C_0 = 6\Lambda_2
      \left(
      \frac{e^{p\left( \frac{\kkk}{\constPV}-1 \right)}}{e^{p\left( \frac{\kkk}{\constPV}-1 \right)}-1}
      \right)^{d/p}
      \end{equation*}
   \end{Theorem}
   {\begin{Remark}
      $\rhd$ The validity of this result depends on the hyperparameter $\kkk$ of our procedure. Using Remark~\ref{rem:theoconcen}, for a given value of $\kkk$, Theorem~\ref{thm:oracle} can be applied for a wide class of models, as soon as the constant $\constPV$, which depends on $H$, $L$, $\alpha$, $b(0_{\R^d})$ and $\lVert\sigma\rVert$ is less than $\kkk$.
   \end{Remark}}
   The oracle inequality allows us to obtain upper bound for the rates of convergence over H\"older balls $\Sigma_d(\boldsymbol{s},\bold L)$ in an adaptive framework.
   \begin{Theorem}\label{thm:adaptive}
   Let $M$ be a nonnegative integer and assume that $K$ is a kernel of order greater than $M$. Set $\mathbf{s}\in(0, M+1]^d$ and $\mathbf{L}\in(0, +\infty)^d$. Asume that $f\in\Sigma_d(\boldsymbol{s},\bold L)$, then under $\HUN$, $\HDEUX$ and $\HTROIS$ we have:
      \begin{equation}\label{toto}
         R_n(\hat f, f) \leq C_1  \left((n\Delta_n)^{-\betaH}
         \log\left((n\Delta_n)^{\beta_H}\right)  \right)^{\gamma(\mathbf{s})}
      \end{equation}
      where
      $$C_1=8\sqrt{\kkk d(d+3/2) p}\LK +7e^{M+1}\Lambda_1+ \Lambda_2+C_0.$$

   \end{Theorem}

   \begin{Remark}
      $\rhd$ This result ensures that the estimator $\hat f$ achieves the rate of convergence obtained in Theorem~\ref{theo:rates} up to a logarithmic factor. Such behavior is well-known for pointwise adaptive estimation, see~\citet{MR1091202,MR1700239,MR3224636} among others.\\
      \noindent
      $\rhd$ If we let the hyperparameter $\kkk$ depend on $n$ (e.g.\@ $\kkk = \log(n)$) then the procedure is also asymptotically adaptive with respect to the values of $L$, $\alpha$, $b(0_{\R^d})$ and $\lVert\sigma\rVert$. In the case $\kkk = \log(n)$, the rate of convergence in \eqref{toto} is multiplied by $(\log n)^{1/2}$.
   \end{Remark}

   \section{On concentration inequalities for fractional SDEs in stationary regime}\label{sec:proofs1}

   \subsection{Sketch of proof of Theorem \ref{main_thm_concentration}}\label{subsec:sketch}
   We denote by $(\Omega,\mathcal{F},\mathproba{P})$ the probability space on which the fBm is defined.
   Let $(\mathcal{F}_t)_{t\geqslant0}$ be the natural filtration associated to the two-sided Brownian motion $(W_t)_{t\in\R}$ induced by the Mandelbrot-Van Ness representation (see \eqref{eq:mandelbrot}). As in \cite{varvenne2019concentration}, let us first introduce the following decomposition. For all $k\in\N$, set
   \begin{equation}\label{def_martingale}
   M_k:=\e[F_X~|~\mathcal{F}_{k}]
   \end{equation}
   {where we recall that $F_X=F(X_{t_1},\dots,X_{t_n})$ and $(X_t)_{t\geq0}$ is the stationary solution of \eqref{modeldiffusion}.}
   Then, we have:
   \begin{equation}\label{eq:decomposition_sum_martingale_inc}
   F_X-\e[F_X]=M_{\lceil t_n\rceil}=\sum_{k=1}^{\lceil t_n\rceil}M_k-M_{k-1}+~\e[F_X~|~\mathcal{F}_0]-\e[F_X].
   \end{equation}
   Our strategy of proof is decomposed as follows, we show that :
   \begin{enumerate}
   \item[(1)] for all $ 1\leq k\leq \lceil t_n\rceil $, there exists $u_k^{(n)}>0$ (deterministic) such that for all $\lambda>0$,
   $$\e\left[\exp\left(\lambda(M_k-M_{k-1})\right)~|~\mathcal{F}_{k-1}\right]\leqslant \exp(\lambda^2 u_k^{(n)})\quad a.s.$$
   and then
   $$\e\left[\exp\left.\left(\lambda\sum_{k=1}^{\lceil t_n \rceil}(M_k-M_{k-1})\right)~\right|~\mathcal{F}_{0}\right]\leqslant \exp\left(\lambda^2 \sum_{k=1}^{\lceil t_n \rceil}u_k^{(n)}\right)\quad a.s.$$
   \item[(2)] there exists $u_0^{(n)}>0$ (deterministic) such that for all $\lambda>0$,
   $$\e\left[\exp\left(\lambda(\e[F_X~|~\mathcal{F}_0]-\e[F_X])\right)\right]\leqslant \exp(\lambda^2 u_0^{(n)})\quad a.s.$$
   \end{enumerate}
   From $(1)$ and $(2)$, we finally get
   $$\e\left[\exp\left(\lambda(F_X-\e[F_X])\right)\right]\leqslant \exp\left(\lambda^2\sum_{k=0}^{\lceil t_n\rceil} u_k^{(n)}\right).$$

   We are thus reduced to study conditional exponential moments in $(1)$ and $(2)$. The related results are given in Proposition \ref{prop:exp_moments_sum_martingale_inc} and \ref{prop:controlfo} (and Theorem \ref{main_thm_concentration} easily follows). \\

   \noindent In order to provide such exponential bounds, we rely on the following key
   lemma (see Lemma 1.5 in Chapter 1 of \cite{rigollet2017high}):
   \begin{Lemma}\label{lem:subgaussian} Let $Z$ be a \textbf{centered} random variable on $\R$ such that there exists $\zeta>0$ such that for all $p\geq2$,
   $$\e[|Z|^p]\leq \zeta^{\frac{p}{2}}p\Gamma\left(\frac{p}{2}\right).$$
   Then, for all $\lambda>0$, we have
   $$\e[\exp(\lambda Z)]\leq \exp(2\lambda^2\zeta).$$
   \end{Lemma}

   \subsection{Part 1: sum of martingale increments}
   In this subsection, our purpose is to prove the following result :

   \begin{Proposition}\label{prop:exp_moments_sum_martingale_inc}
   Assume $\HUN$ and $\HDEUX$. Let $H\in(0,1)$.
   There exists {$K=K(H,L,\alpha,\|\sigma\|)>0$} such that for all $\lambda>0$,
   \begin{equation}
   \e\left[\exp\left.\left(\lambda\left(\sum_{k=1}^{\lceil t_n\rceil}M_k-M_{k-1}\right)\right)~\right|~\mathcal{F}_{0}\right]\leqslant \exp\left(K\lambda^2\|F\|^2_{\rm Lip}\sum_{k=1}^{\lceil t_n\rceil}\psi^2_{n,k}\right)\quad a.s.
   \end{equation}
   where $$\psi_{n,k}:=\left\lfloor\frac{k}{\Delta_n}\right\rfloor- \left\lceil\frac{k-1}{\Delta_n}\right\rceil+\sum_{i=\left\lceil\frac{k}{\Delta_n}\right\rceil}^{n}u_i^{H-3/2},$$
   $u_i:=t_i-k+1=i\Delta_n-k+1$ and $M_k$ is defined by~\eqref{def_martingale} .\\
   Moreover, there exists $c=c_{}(H)>0$ such that
   $$\sum_{k=1}^{\lceil t_n\rceil}\psi^2_{n,k}\leq c\left\{\begin{array}{lll}
   n\Delta_n^{-1}&\text{ if }& H\in(0,1/2)\\
   n^{2H}\Delta_n^{2H-2}&\text{ if }& H\in(1/2,1).\end{array}\right.$$
   \end{Proposition}
   \begin{Remark}\label{rem:constant1} Following carefully the constants in the proof of this proposition, one easily checks that $(L,\alpha,s)\mapsto K(H,L,\alpha,s)$ is bounded on every compact set of $\ER_+\times\ER_+^*\times\ER_+$.
   \end{Remark}
   Through equation \eqref{modeldiffusion} and the fact that $b$ is Lipschitz continuous, for all $t\geqslant0$,~ $Y_{t}$ can be seen as a functional of the time $t$, the initial condition $X_0$ and the Brownian motion $(W_s)_{s\leq t}$. Denote by $\Phi:\R_+\times\R^d\times\mathcal{C}(\R,\R^d)\to\R^d$ this functional, we then have
   \begin{equation}
   \forall t\geqslant0,\quad X_{t}:=\Phi_t(X_0,(W_s)_{s\leq t}).
   \end{equation}

   \noindent Now, let $k\geqslant1$, we have
   \begin{align}\label{eq:majo_M}
   &|M_k-M_{k-1}|\nonumber\\
   &=|\e[F_X|\F_k]-\e[F_X|\F_{k-1}]|\nonumber\\
   &\leqslant\|F\|_{\rm Lip}\int_{\Omega}\sum_{t_i\geq k-1}\left|\Phi_{t_i}\left(X_0,(W_s)_{s\leq k}\sqcup\tilde{w}_{[k,t_i]}\right)-\Phi_{t_i}\left(X_0,(W_s)_{s\leq k-1}\sqcup\tilde{w}_{[k-1,t_i]}\right)\right|\p_W({\rm d}\tilde{w}).
   \end{align}

   Let us introduce now some notations. First, for all $t\geqslant k-1$ set $u:=t-k+1$, then for all $u\geqslant0$, we define
   $$
   Y_u:=\left\{\begin{array}{lll}
   \Phi_{u+k-1}\left(X_0,(W_s)_{s\leq k}\sqcup(\tilde{w}_s)_{s\in[k,u+k-1]}\right) & \text{if } u\geqslant1 \\
   \Phi_{u+k-1}\left(X_0,(W_s)_{s\leq u+k-1}\right) & \text{otherwise},
   \end{array}\right.
   $$
   and
   $$
   \tilde{Y}_u:=\Phi_{u+k-1}\left(X_0,(W_s)_{s\leq k-1}\sqcup(\tilde{w}_s)_{s\in[k-1,u+k-1]}\right).
   $$
   By using equation \eqref{modeldiffusion}, we then have
   \begin{align}\label{eq:sde_Y_tildeY}
   Y_{u}-\tilde{Y}_u
   =\int_{0}^u b(Y_s)-b(\tilde{Y}_s){\rm d}s+c_H\sigma\int_{0}^{1\wedge u}(u-s)^{H-\frac{1}{2}}{\rm d}(W^{(k)}-\tilde{w}^{(k)})_s.
   \end{align}
   where we have set $(W^{(k)}_s)_{s\geqslant 0}:=(W_{s+k-1}-W_{k-1})_{s\geqslant0}$ which is a Brownian motion independent from $\F_{k-1}$ and $(\tilde{w}^{(k)}_s)_{s\geqslant0}:=(\tilde{w}_{s+k-1}-\tilde{w}_{k-1})_{s\geqslant0}$.\\

   Finally, we have the following inequality for all $k\geq1$:
   \begin{equation}\label{eq:majo_M_Y}
   |M_k-M_{k-1}|\leq \|F\|_{\rm Lip}\int_{\Omega}\sum_{u_i\geq 0}\left|Y_{u_i}-\tilde{Y}_{u_i}\right|\p_W({\rm d}\tilde{w})
   \end{equation}
   where $u_i:=t_i-k+1=i\Delta_n-k+1$.\\

   In the next section, we proceed to a control of the quantity $|Y_u-\tilde{Y}_u|$.

   \subsubsection{Control lemma}

   \begin{Lemma}\label{lem:control_Y_Ytilde} We have the two following inequalities:
   \begin{itemize}
   \item[$(i)$] for all $u\in[0,1]$, there exists {$K=K(H,L,\alpha,\|\sigma\|)>0$} such that,
   \begin{equation*}
   |Y_{u}-\tilde{Y}_{u}|\leqslant {K}\sup\limits_{v\in[0,2]}\left|G_v^{(k)}(W-\tilde{w})\right|
   \end{equation*}
   where $$G_v^{(k)}(W-\tilde{w}):=\int_{0}^{1\wedge v}(v-s)^{H-\frac{1}{2}}{\rm d}(W^{(k)}-\tilde{w}^{(k)})_s,$$
   \item[$(ii)$] for all $u\geq1$, there exists {$K=K(H,L,\alpha,\|\sigma\|)>0$} such that,
   \begin{equation*}
   |Y_{u}-\tilde{Y}_{u}|\leqslant {K}~u^{H-\frac{3}{2}}\sup\limits_{s\in[0,1]}|W_s^{(k)}-\tilde{w}^{(k)}_s|
   \end{equation*}
   \end{itemize}
   \end{Lemma}

   \begin{proof}
   $\rhd$ First case: let $u\in[0,1]$.\\
   By the triangle inequality and assumption $\HDEUX$, we have in \eqref{eq:sde_Y_tildeY}
   \begin{equation*}
   |Y_{u}-\tilde{Y}_{u}|\leqslant L\int_{0}^u |Y_s-\tilde{Y}_s|{\rm d}s+c_H\|\sigma\|\left|\int_{0}^{1\wedge u}(u-s)^{H-\frac{1}{2}}{\rm d}(W^{(k)}-\tilde{w}^{(k)})_s\right|
   \end{equation*}
   Then, from Gronwall's lemma, we deduce the following
   \begin{align}\label{eq:proof4}
   &|Y_{u}-\tilde{Y}_{u}|\nonumber\\
   &\leqslant c_H\|\sigma\|\left|\int_{0}^{1\wedge u}(u-s)^{H-\frac{1}{2}}{\rm d}(W^{(k)}-\tilde{w}^{(k)})_s\right| \nonumber\\
   &\quad\quad\quad\quad\quad\quad+c_H\|\sigma\|L\int_{0}^u\left|\int_{0}^{1\wedge v}(v-s)^{H-\frac{1}{2}}{\rm d}(W^{(k)}-\tilde{w}^{(k)})_s\right|e^{L(u-v)}{\rm d}v\nonumber\\
   &\leqslant c_H\|\sigma\|\sup\limits_{v\in[0,2]}\left|\int_{0}^{1\wedge v}(v-s)^{H-\frac{1}{2}}{\rm d}(W^{(k)}-\tilde{w}^{(k)})_s\right|\left(1+\left[-e^{L(u-v)}\right]_{0}^u\right)\nonumber\\
   &\leqslant c_H\|\sigma\| e^{2L}\sup\limits_{v\in[0,2]}\left|\int_{0}^{1\wedge v}(v-s)^{H-\frac{1}{2}}{\rm d}(W^{(k)}-\tilde{w}^{(k)})_s\right|.
   \end{align}
   and Lemma \ref{lem:control_Y_Ytilde} is shown for $u\in[0,1]$.\\

   $\rhd$ Second case: let $u\geqslant1$.\\
   First, if $u\in[1,2]$, we have $|Y_u-\tilde{Y}_u|\leqslant K_H\sup\limits_{s\in[0,1]}|W^{(k)}_s-\tilde{w}^{(k)}_s|$ by the first part of this proof combined with the following inequality :
   \begin{align*}
   &\left|\int_{0}^{1\wedge v}(v-s)^{H-\frac{1}{2}}{\rm d}(W^{(k)}-\tilde{w}^{(k)})_s\right|\\
   &=\left|(v-1\wedge v)^{H-\frac{1}{2}}(W^{(k)}_{1\wedge v}-\tilde{w}^{(k)}_{1\wedge v})+(H-1/2)\int_{0}^{1\wedge v}(v-s)^{H-\frac{3}{2}}(W^{(k)}_s-\tilde{w}^{(k)}_s){\rm d}s\right|\\
   &\leqslant K_H\sup\limits_{s\in[0,1]}|W^{(k)}_s-\tilde{w}^{(k)}_s|.
   \end{align*}
   Now, let us treat the case $u\geq2$.
   In the following inequalities, we use assumption $\HUN$ on the function $b$ and the elementary Young inequality $\langle a,b\rangle\leqslant\frac{1}{2}\left(\varepsilon|a|^2+\frac{1}{\varepsilon}|b|^2\right)$ with $\varepsilon=2\alpha$.
   \begin{align}
   &\frac{{\rm d}}{{\rm d}u}|Y_{u}-\tilde{Y}_{u}|^2\nonumber\\
   &=2\langle Y_u-\tilde{X}_u,b(Y_u)-b(\tilde{Y}_u)\rangle+c_H(2H-1)\langle Y_u-\tilde{Y}_u,~\sigma\int_0^1(u-s)^{H-\frac{3}{2}}{\rm d}(W^{(k)}-\tilde{w}^{(k)})_s\rangle\nonumber\\
   &\leqslant -2\alpha|Y_{u}-\tilde{Y}_{u}|^2+\alpha|Y_{u}-\tilde{Y}_{u}|^2+\frac{c_H^2(2H-1)^2\|\sigma\|^2}{4\alpha}\left|\int_0^1(u-s)^{H-\frac{3}{2}}{\rm d}(W^{(k)}-\tilde{w}^{(k)})_s\right|^2\nonumber\\
   &\leqslant -\alpha|Y_{u}-\tilde{Y}_{u}|^2+\frac{c_H^2(2H-1)^2\|\sigma\|^2}{4\alpha}\left|\int_0^1(u-s)^{H-\frac{3}{2}}{\rm d}(W^{(k)}-\tilde{w}^{(k)})_s\right|^2.\nonumber
   \end{align}
   We then apply Gronwall's lemma to obtain
   \begin{equation}\label{eq:proof2}
   |Y_{u}-\tilde{Y}_{u}|^2\leqslant e^{-\alpha(u-2)}|Y_2-\tilde{Y}_2|^2+\alpha_H\int_{2}^ue^{-\alpha(u-v)}\left|\int_0^1(v-s)^{H-\frac{3}{2}}{\rm d}(W^{(k)}-\tilde{w}^{(k)})_s\right|^2{\rm d}v
   \end{equation}
   with $\alpha_H:=\frac{c_H^2(2H-1)^2\|\sigma\|^2}{4\alpha}$.\\
   Now, we set $~\varphi_k(v):=\int_0^1(v-s)^{H-\frac{3}{2}}{\rm d}(W^{(k)}-\tilde{w}^{(k)})_s~$ and we apply an integration by parts to $\varphi_k$ taking into account that $W^{(k)}_0=\tilde{w}^{(k)}_0=0$:
   \begin{align*}
   \varphi_k(v)=(v-1)^{H-\frac{3}{2}}(W^{(k)}_1-\tilde{w}^{(k)}_1)-(3/2-H)\int_0^1(v-s)^{H-\frac{5}{2}}(W^{(k)}_s-\tilde{w}^{(k)}_s){\rm d}s.
   \end{align*}
   And then
   \begin{align}\label{eq:proof1}
   &\int_{2}^ue^{-\alpha(u-v)}\left|\varphi_k(v)\right|^2{\rm d}v\nonumber\\
   &\leqslant 2|W^{(k)}_1-\tilde{w}^{(k)}_1|^2\int_{2}^u(v-1)^{2H-3}e^{-\alpha(u-v)}{\rm d}v\nonumber\\
   &\quad\quad\quad+2(3/2-H)^2\int_{2}^ue^{-\alpha(u-v)}\left|\int_0^1(v-s)^{H-\frac{5}{2}}(W^{(k)}_s-\tilde{w}^{(k)}_s){\rm d}s\right|^2{\rm d}v\nonumber\\
   &\leqslant 2|W^{(k)}_1-\tilde{w}^{(k)}_1|^2\int_{2}^u(v-1)^{2H-3}e^{-\alpha(u-v)}{\rm d}v\nonumber\\
   &\quad\quad\quad+2(3/2-H)^2\left|\int_0^1(W^{(k)}_s-\tilde{w}^{(k)}_s){\rm d}s\right|^2\int_{2}^u(v-1)^{2H-5}e^{-\alpha(u-v)}{\rm d}v
   \end{align}

   \begin{Lemma}\label{lem:control_integral} Let $\alpha,\beta>0$. Then, for all $u\geqslant2$,
   \begin{equation*}
   e^{-\alpha u}\int_{2}^u (v-1)^{-\beta}e^{\alpha v}{\rm d}v\leqslant\max\left(e^{-\alpha(u-2)}(u-1),~(u-1)^{-\beta+1}\right).
   \end{equation*}
   \end{Lemma}

   In the right hand side of \eqref{eq:proof1}, we apply an integration by parts on the first term and then we use Lemma \ref{lem:control_integral}:
   \begin{align*}
   &\int_{2}^u(v-1)^{2H-3}e^{-\alpha(u-v)}{\rm d}v\\
   &=\frac{1}{\alpha}\left((u-1)^{2H-3}-e^{-\alpha(u-2)}+(3-2H)e^{-\alpha u}\int_{2}^u (v-1)^{2H-4}e^{\alpha v}{\rm d}v\right)\nonumber\\
   &\leqslant\frac{1}{\alpha}\left((u-1)^{2H-3}-e^{-\alpha(u-2)}+(3-2H)\max\left(e^{-\alpha(u-2)}(u-1),~(u-1)^{2H-3}\right)\right)\nonumber\\
   &\leqslant C'_H(u-1)^{2H-3}
   \end{align*}
   where $C'_H>0$ is some constant. Finally, by using Lemma \ref{lem:control_integral} also on the second term in \eqref{eq:proof1}, we finally get the existence of a constant $C''_H>0$ such that:
   \begin{equation}\label{eq:proof3}
   \int_{2}^ue^{-\alpha(u-v)}\left|\varphi_k(v)\right|^2{\rm d}v\leqslant C''_H(u-1)^{2H-3}\sup\limits_{s\in[0,1]}|W^{(k)}_s-\tilde{w}^{(k)}_s|^2.
   \end{equation}
   Now, putting inequality \eqref{eq:proof3} into \eqref{eq:proof2} and taking the square root, we finally get:
   \begin{equation}\label{eq:proof5}
   |Y_{u}-\tilde{Y}_{u}|\leqslant e^{-\frac{\alpha}{2}(u-2)}|Y_2-\tilde{Y}_2|+C^{(3)}_H(u-1)^{H-\frac{3}{2}}\sup\limits_{s\in[0,1]}|W^{(k)}_s-\tilde{w}^{(k)}_s|.
   \end{equation}
   On the one hand, we can note that $e^{-\frac{\alpha}{2}(u-2)}\leqslant C' u^{H-\frac{3}{2}}$ for all $u\geqslant2$. On the other hand, we have $|Y_2-\tilde{Y}_2|\leqslant K_H\sup\limits_{s\in[0,1]}|W^{(k)}_s-\tilde{w}^{(k)}_s|$.

   These two facts combined with \eqref{eq:proof5} conclude the proof.

   \end{proof}

   \subsubsection{Conditional exponential moments of the martingale increments}

   \begin{Proposition}\label{prop:moments_martingale_inc} {Assume $\HUN$ and $\HDEUX$. Let $H\in(0,1)$.}
   There exists {$K=K(H,L,\alpha,\|\sigma\|)>0$ and} $\zeta>0$ such that for all $k\in\N^*$ and for all $p\geqslant2$,
   \begin{equation}
   \e[|M_k-M_{k-1}|^p|\mathcal{F}_{k-1}]\leqslant C^p\|F\|^p_{\rm Lip}\psi^p_{n,k}\zeta^{p/2}p\Gamma\left(\frac{p}{2}\right)\quad a.s.
   \end{equation}
   where $$\psi_{n,k}:=\left\lfloor\frac{k}{\Delta_n}\right\rfloor- \left\lceil\frac{k-1}{\Delta_n}\right\rceil+\sum_{i=\left\lceil\frac{k}{\Delta_n}\right\rceil}^{n}u_i^{H-3/2},$$
   {$u_i:=t_i-k+1=i\Delta_n-k+1$ and $M_k$ is defined by \eqref{def_martingale}.}
   \end{Proposition}

   \begin{proof}
   Let $k\in\N^*$ and $p\geq2$. By combining inequality \eqref{eq:majo_M_Y} with the technical lemma \ref{lem:control_Y_Ytilde}, we immediately get that there exists {$K=K(H,L,\alpha,\|\sigma\|)>0$} such that:
   \begin{align*}
   &|M_k-M_{k-1}|^p\\
   &\leq {K}^p\|F\|^p_{\rm Lip}\psi^p_{n,k}\left(\int_\Omega\sup\limits_{v\in[0,2]}\left|G_v^{(k)}(W-\tilde{w})\right|+\sup\limits_{s\in[0,1]}|W_s^{(k)}-\tilde{w}^{(k)}_s|\p_W({\rm d}\tilde{w})\right)^p
   \end{align*}
   The end of the proof consists in proving that
   \begin{align*}
   &\e\left[\left.\left(\int_\Omega\sup\limits_{v\in[0,2]}\left|G_v^{(k)}(W-\tilde{w})\right|+\sup\limits_{s\in[0,1]}|W_s^{(k)}-\tilde{w}^{(k)}_s|\p_W({\rm d}\tilde{w})\right)^p~\right|~\F_{k-1}\right]\\
   &\leq \zeta^{p/2}p\Gamma\left(\frac{p}{2}\right).
   \end{align*}

   \noindent We have
   \begin{align}
   &\e\left[|M_k-M_{k-1}|^p|\F_{k-1}\right]\nonumber\\
   &\leq 2^{p-1}\e\left[\left.\left(\int_\Omega\sup\limits_{v\in[0,2]}\left|G_v^{(k)}(W-\tilde{w})\right|\p_W({\rm d}\tilde{w})\right)^p\right|~\F_{k-1}\right]\nonumber\\
   &\quad\quad\quad+2^{p-1}\e\left[\left.\left(\int_\Omega\sup\limits_{s\in[0,1]}|W_s^{(k)}-\tilde{w}^{(k)}_s|\p_W({\rm d}\tilde{w})\right)^p~\right|~\F_{k-1}\right]\nonumber\\
   &\leq 2^{p-1}\e\left[\left(\int_\Omega\sup\limits_{v\in[0,2]}\left|G_v^{(k)}(W-\tilde{w})\right|\p_W({\rm d}\tilde{w})\right)^p\right]\nonumber\\
   &\quad\quad\quad+2^{p-1}\e\left[\left(\int_\Omega\sup\limits_{s\in[0,1]}|W_s^{(k)}-\tilde{w}^{(k)}_s|\p_W({\rm d}\tilde{w})\right)^p\right]\label{eq:1proof_prop_martingale_inc}
   \end{align}
   where the last inequality is obtained by using that $W^{(k)}=(W_{s+k-1}-W_{k-1})_{s\geq0}$ is independent from $\F_{k-1}$.
   Now, if we denote by $\F^{(k)}$ the natural filtration associated to $W^{(k)}$, then the right hand side terms of \eqref{eq:1proof_prop_martingale_inc} are just expectations of conditional expectations with respect to $\F^{(k)}_1$, so we finally get
   \begin{align*}
   \e\left[|M_k-M_{k-1}|^p|\F_{k-1}\right]&\leq 2^{p-1}\e\left[\sup\limits_{v\in[0,2]}\left|G_v^{(k)}(W-\tilde{W})\right|^p\right]\nonumber\\
   &\quad\quad\quad+2^{p-1}\e\left[\sup\limits_{s\in[0,1]}|W_s^{(k)}-\tilde{W}^{(k)}_s|^p\right].
   \end{align*}
   Since $W^{(k)}$ and $\tilde{W}^{(k)}$ are i.i.d. and have the same law as $W^{(1)}$, we can replace $W^{(k)}-\tilde{W}^{(k)}$ by $\sqrt{2}W^{(1)}$, which gives
   \begin{align*}
   &\e\left[|M_k-M_{k-1}|^p|\F_{k-1}\right]\\
   &\leq 2^{p-1}\sqrt{2}\e\left[\sup\limits_{v\in[0,2]}\left|G_v^{(1)}(W)\right|^p\right]+2^{p-1}\sqrt{2}\e\left[\sup\limits_{s\in[0,1]}|W_s^{(1)}|^p\right].
   \end{align*}
   To conclude the proof, we only have to prove that $\sup_{v\in[0,2]}\left|G_v^{(1)}(W)\right|$ and $\sup_{s\in[0,1]}|W_s^{(1)}|$ are sub-Gaussian. The proof of this result follows the lines of \cite{varvenne2019concentration} Appendices A and B and we leave it to the patient reader.
   \end{proof}

   \noindent With Lemma \ref{lem:subgaussian} in hand, we finally get :

   \begin{Proposition}\label{prop:exp_moments_martingale_inc} {Assume $\HUN$ and $\HDEUX$. Let $H\in(0,1)$.}
   Let $k\in\N^*$. There exists {$K=K(H,L,\alpha,\|\sigma\|)>0$} (independent of $k$) such that for all $\lambda>0$,
   \begin{equation}
   \e[\exp(\lambda(M_k-M_{k-1}))~|~\mathcal{F}_{k-1}]\leqslant \exp\left({K}\lambda^2\|F\|^2_{\rm Lip}\psi^2_{n,k}\right)\quad a.s.
   \end{equation}
   where $$\psi_{n,k}:=\left\lfloor\frac{k}{\Delta_n}\right\rfloor- \left\lceil\frac{k-1}{\Delta_n}\right\rceil+\sum_{i=\left\lceil\frac{k}{\Delta_n}\right\rceil}^{n}u_i^{H-3/2},$$
   {$u_i:=t_i-k+1=i\Delta_n-k+1$ and $M_k$ is defined by \eqref{def_martingale}.}

   \end{Proposition}

   \subsubsection{Proof of Proposition \ref{prop:exp_moments_sum_martingale_inc}}

   \begin{proof} The inequality on the conditional Laplace transform of $\sum_{k=1}^{\lceil t_n\rceil}M_k-M_{k-1}$ easily follows from Proposition \ref{prop:exp_moments_martingale_inc}. We thus have to prove the bound on $\sum_{k=1}^{\lceil t_n\rceil}\psi^2_{n,k}$.\\
   Let us begin by the estimation of $\psi_{n,k}$ for all $k\in\{1,\dots,\lceil t_n\rceil\}$.
   First, we easily get that
   \begin{equation}\label{eq:estim_psi_1}
   \left\lfloor\frac{k}{\Delta_n}\right\rfloor- \left\lceil\frac{k-1}{\Delta_n}\right\rceil\leq \Delta_n^{-1}.
   \end{equation}
   Secondly, let us consider the second part of $\psi_{n,k}$, we have
   \begin{align}\label{eq:estim_psi_2}
   &\sum_{i=\left\lceil\frac{k}{\Delta_n}\right\rceil}^{n}(i\Delta_n-k+1)^{H-3/2}\nonumber\\
   &\leq 1+\int_{\left\lceil k\Delta_n^{-1}\right\rceil}^n(t\Delta_n-k+1)^{H-3/2}{\rm d}t\nonumber\\
   &=1+\Delta_n^{H-3/2}\left[\frac{(t-(k-1)\Delta_n^{-1})^{H-1/2}}{H-1/2}\right]_{\left\lceil k\Delta_n^{-1}\right\rceil}^n\nonumber\\
   &=1+\frac{\Delta_n^{H-3/2}}{H-1/2}\left[(n-(k-1)\Delta_n^{-1})^{H-1/2}-\left(\left\lceil k\Delta_n^{-1}\right\rceil-(k-1)\Delta_n^{-1}\right)^{H-1/2}\right]\nonumber\\
   &\leq 1+\Delta_n^{H-3/2}\times\left\{\begin{array}{lll}
   \frac{1}{H-1/2}\left(n-(k-1)\Delta_n^{-1}\right)^{H-1/2} & \text{ if } & H>1/2\\
   \frac{1}{1/2-H}\left(\left\lceil k\Delta_n^{-1}\right\rceil-(k-1)\Delta_n^{-1}\right)^{H-1/2}& \text{ if } & H<1/2
   \end{array}\right.\nonumber\\
   &\leq 1+\Delta_n^{H-3/2}\times\left\{\begin{array}{lll}
   \frac{1}{H-1/2}\left(n-(k-1)\Delta_n^{-1}\right)^{H-1/2} & \text{ if } & H>1/2\\
   \frac{1}{1/2-H}\left(1+\Delta_n^{-1}\right)^{H-1/2}& \text{ if } & H<1/2.
   \end{array}\right.
   \end{align}
   From \eqref{eq:estim_psi_1} and \eqref{eq:estim_psi_2}, we thus deduce that there exists $c_1=c_1(H)>0$ such that
   \begin{equation}
   \psi_{n,k}\leq\Delta_n^{-1}+1+c_1\left\{\begin{array}{lll}
   \Delta_n^{-1}\left(t_n-(k-1)\right)^{H-1/2} & \text{ if } & H>1/2\\
   \Delta_n^{-1}& \text{ if } & H<1/2.
   \end{array}\right.
   \end{equation}
   We can now move on the estimation of $\sum_{k=1}^{\lceil t_n\rceil}\psi^2_{n,k}$.
   From the inequality above, it follows that there exists $c_2=c_2(H)>0$ such that
   \begin{align}\label{eq:estim_psi_3}
   \sum_{k=1}^{\lceil t_n\rceil}\psi^2_{n,k}
   \leq c_2(&\Delta_n^{-2}+1)\lceil t_n\rceil\nonumber\\
   &+c_2\left\{\begin{array}{lll}
   \Delta_n^{-2}\sum_{k=1}^{\lceil t_n\rceil}\left(t_n-(k-1)\right)^{2H-1} & \text{ if } & H>1/2\\
   \Delta_n^{-2}\lceil t_n\rceil& \text{ if } & H<1/2.
   \end{array}\right.
   \end{align}
   It remains to estimate $\sum_{k=1}^{\lceil t_n\rceil}\left(t_n-(k-1)\right)^{2H-1}$ when $H>1/2$. It is readily checked that
   \begin{equation}\label{eq:estim_psi_4}
   \sum_{k=1}^{\lceil t_n\rceil}\left(t_n-(k-1)\right)^{2H-1}\leq\sum_{k=1}^{\lceil t_n\rceil}\left(\lceil t_n\rceil-(k-1)\right)^{2H-1}=\sum_{k=1}^{\lceil t_n\rceil}k^{2H-1}\leq\lceil t_n\rceil^{2H}.
   \end{equation}
   Finally, from \eqref{eq:estim_psi_3} and \eqref{eq:estim_psi_4}, we get the existence of $c_3=c_3(H)>0$ and $c_4=c_4(H)>0$ such that
   \begin{align*}
   \sum_{k=1}^{\lceil t_n\rceil}\psi^2_{n,k}&\leq c_3\left\{\begin{array}{lll}
   \Delta_n^{-2}\lceil t_n\rceil+\Delta_n^{-2}\lceil t_n\rceil^{2H} & \text{ if } & H>1/2\\
   \Delta_n^{-2}\lceil t_n\rceil& \text{ if } & H<1/2
   \end{array}\right.\\
   &\leq c_4\left\{\begin{array}{lll}
   \Delta_n^{-2}\lceil t_n\rceil^{2H} & \text{ if } & H>1/2\\
   \Delta_n^{-2}\lceil t_n\rceil & \text{ if } & H<1/2.
   \end{array}\right.
   \end{align*}
   This concludes the proof since $t_n=n\Delta_n$.
   \end{proof}

   \subsection{Part 2: $\e[F_X~|~\mathcal{F}_0]-\e[F_X]$}
   We now turn to the bound of $\e[F_X~|~\mathcal{F}_0]-\e[F_X]$.
   First, let us remark that
   $$\e[\e[F_X~|~\mathcal{F}_0]-\e[F_X]]=0$$
   so that we can use again Lemma \ref{lem:subgaussian} in order to deduce exponential bounds. The related result is stated in Proposition \ref{prop:controlfo}.\\

   \noindent We introduce notations related to the conditioning with respect to $\mathcal{F}_0$. Let $(W_t)_{t\in\ER}$ denote the two-sided Brownian Motion induced by   Mandelbrot-Van Ness representation (see \eqref{eq:mandelbrot}) and set $W^{-}=(W_t)_{t\le 0}$.   For  $\varepsilon\in(0,1/2)$ and $\varepsilon'>0$, set
   \begin{multline}\label{eq:def-omega-minus}
   \Omega_-:=\Big\{  w:(-\infty,0]\to \R^d; \, w(0)=0\, ,\,  \lim_{t\rightarrow-\infty} \frac{w(t)}{|t|^{\frac{1}{2}+\varepsilon'}}=0 \ \text{and}\\
   w \text{ is $(\frac{1}{2}-\varepsilon)$-H\"older continuous on compact intervals}  \Big\}.
   \end{multline}
   Owing to some classical properties on the Wiener process, this subspace is of Wiener measure $1$ for any fixed
   $\varepsilon\in(0,1/2)$ and $\varepsilon'>0$. In other words, $\mathproba{P}_{W^{-}}(\Omega_-)=1$. Then, for any $w\in \Omega_-$,
   $${\cal L}((B_t^H)_{t\ge0}|W^{-}=w)={\cal L}((Z_t+ D_t(w))_{t\ge 0})$$
   where $Z_0=D_0(w)=0$ and for all $t>0$,
   $$ Z_t=c_H\int_0^t (t-s)^{H-\frac{1}{2}} -(-s)^{H-\frac{1}{2}}  dW_s$$
   and
   $$D_t(w)=\alpha_H\int_{-\infty}^0 (t-s)^{H-\frac{1}{2}} -(-s)^{H-\frac{1}{2}}  dw_s.$$
    $(Z_t)_{t\ge0}$ and $(D_t(w))_{t\ge0}$ are continuous processes on $[0,+\infty)$ (see Lemma \ref{propDmoins} below) and for any $w\in \Omega_-$, the (additive) SDE
    \begin{equation}\label{eq:conditionededs}
     dY_t=b(Y_t) dt+\sigma (dZ_t+dD_t^{w})
     \end{equation}
    has a unique solution denoted by $(X_t^{x,w})_{t\ge0}$. \\

   \noindent  Since $F$ is Lipschitz continuous with respect to $d_n$, we can also remark that
   \begin{equation}\label{eq:controlfxfo}
   |\e[F_X~|~\mathcal{F}_0]-\e[F_X]|\le \|F\|_{\rm Lip} \sum_{k=1}^n \int_{\R^d\times\Omega_-} |X_{t_k}^{X_0,W^{-}}-X_{t_k}^{y,w}|\nu(dy,dw),
   \end{equation}
   where $\nu={\cal L}(X_0,W^-)$ is an initial condition for the dynamical system which is such that the process is stationary.
   This involves that we will use bounds on $|X_{t}^{x,W^{-}}-X_{t}^{y,w}|$ to deduce the result for $\e[F_X~|~\mathcal{F}_0]-\e[F_X]$.  To this end,  we first state a technical result about  $D(w)$:
    \begin{Lemma}\label{propDmoins} Let $w\in \Omega^{-}$ with $\varepsilon\in(0,H)$ and $\varepsilon'\in(0,1-H)$. Then, $(D_{t}(w))_{t\ge0}$ is continuous on $[0,\infty)$ and
   differentiable on $(0,+\infty)$. Furthermore, for any $\delta_1\in(0,1-H-\varepsilon']$ and $\delta_2\in \le H-1-\varepsilon$, there exist some positive constants $C_{\delta_1}$ and $C_{\delta_2}$ such that for any $t>0$,
   \begin{equation}\label{dprimetphieps}
   \begin{split}
   & |D'_t(w)|\le  \Phi_{\varepsilon,\varepsilon'}(w) \left(C_{\delta_1}t^{-\delta_1} 1_{t>1}+C_{\delta_2} t^{-\delta_2} 1_{t\in(0,1]}\right)\\
     &\textnormal{with}\; \Phi_{\varepsilon,\varepsilon'}(w)=
     \sup_{r\in(0,1]}\frac{|w(-r)|}{r^{\frac{1}{2}-\varepsilon}}+\sup_{r>1}\frac{|w(-r)|}{r^{\frac{1}{2}+\varepsilon'}}.
    \end{split}
    \end{equation}
   \end{Lemma}
   \begin{proof}
   Let $w\in \Omega^{-}$. By an integration by parts, one checks that the process $(D_{t}(w))$
   is well-defined for any $t>0$ and admits the following alternative representation:
   \begin{equation}\label{ipp-d-x-minus}
   { D}_t({w})=\alpha_H\left(H-\frac 12\right)\int_{-\infty}^0 \left((t-r)^{H-\frac 32}-(-r)^{H-\frac 32}\right) w(r)\, dr\quad \textnormal{if $t\in (0,1]$}
   \end{equation}
   It easily follows that $(D_t(w))_{t>0}$ is smooth on $(0,+\infty]$ and that for all $t>0$,
   \begin{equation}
   { D}'_t({w})=\alpha_H\left(H-\frac 12\right)\left(H-\frac32\right)\int_{-\infty}^0 (t-r)^{H-\frac 52}w(r)\, dr.
   \end{equation}
   On the one hand, for any $\bar{\varepsilon}\in [\varepsilon',1-H)$,
   \begin{align*}
   |\int_{-\infty}^{-1} (t-r)^{H-\frac 52}w(r)&\, dr|\le \sup_{r>1}\frac{|w(-r)|}{r^{\frac{1}{2}+\bar{\varepsilon}}}
   \int_{-\infty}^{-1} (t-r)^{H-\frac 52} r^{\frac{1}{2}+\bar{\varepsilon}} dr\\
   &\le \Phi_{\varepsilon,\varepsilon'}(w)\int
   _{-\infty}^{-1} (t-r)^{H-2+\bar{\varepsilon}} dr\le \Phi_{\varepsilon,\varepsilon'}(w)(1+t)^{H-2+\bar{\varepsilon}}.
   \end{align*}
   On the other hand, for any $\tilde{\varepsilon}\ge\varepsilon$,
   $$\int_{-1}^0 (t-r)^{H-\frac 52}w(r)\, dr\le \sup_{r\in(0,1]}\frac{|w(-r)|}{r^{\frac{1}{2}-\tilde{\varepsilon}}}
   \int_{-1}^{0} (t-r)^{H-2-\tilde{\varepsilon}} dr\le \Phi_{\varepsilon,\varepsilon'}(w) t^{H-1-\tilde{\varepsilon}}.$$

   Inequality \eqref{dprimetphieps} easily follows from what precedes.
   In particular,  since $\varepsilon\in(0,H)$, $t\mapsto D'_t(w)$ is integrable near $0$ and hence, $t\mapsto D_t(w)$ is continuous on $\ER_+$.
   \end{proof}


   \noindent In view of  \eqref{eq:controlfxfo}, we now provide a control of the evolution of two paths of the fractional SDE \eqref{eq:conditionededs} starting from initial conditions $({x,\tilde{w}})$ and $({y,{w}})$.
   \begin{Lemma}\label{lem:conttrajxw} Suppose that assumptions $\HUN$ and $\HDEUX$ are in force. Let $w$ and $\tilde{w}$ belong to $\Omega^{-}$ with $\varepsilon \in(0,H)$ and $\varepsilon'\in(0,1-H)$.  Then, there exist some positive constants {$C_1=C_1(H,L,\alpha,\|\sigma\|)$ and $C_2=C_2(H,L,\alpha,\|\sigma\|)$} such that for every $t\ge0$,
    $$|X_{t}^{x,\widetilde{w}}-X_{t}^{y,w}|^2\le  C_1e^{-{\alpha} t}  |x-y|^2+ C_2\Phi_{\varepsilon,\varepsilon'}(\tilde{w}-w)(t\vee 1)^{2H-2+\varepsilon'}$$
    where $\Phi_{\varepsilon,\varepsilon'}$ is defined by \eqref{dprimetphieps}.
   \end{Lemma}
   \begin{proof} For two paths ${w}$ and $\widetilde{w}$ in $\Omega^{-}$,
   $$ B_t^{\widetilde{w}}-B_t^{w}= D_t(\widetilde{w}-w).$$
   Thus, for any $t_0>0$,
   $$ X_t^{x,\widetilde{w}}-X_t^{y,w}=
   X_{t_0}^{x,\widetilde{w}}-X_{t_0}^{y,w}+\int_{t_0}^t b(X_s^{x,\widetilde{w}})- b(X_s^{x,{w}^-}) ds+\sigma D_t(\widetilde{w}-w).$$
   As a consequence,
   \begin{align*}
   & e^{\alpha t} |X_t^{x,\widetilde{w}}-X_t^{y,w}|^2=e^{\alpha t_0} |X_{t_0}^{x,\widetilde{w}}-X_{t_0}^{y,w}|^2\\
    &+\int_0^t 2e^{\alpha s}\langle X_s^{x,\widetilde{w}}-X_s^{x,{w}}), \sigma D'_s (\widetilde{w}-w)\rangle ds\\
    &+\int_0^t e^{\alpha s}\left(\alpha |X_s^{x,\widetilde{w}}-X_s^{y,w}|^2+2\langle X_s^{x,\widetilde{w}}-X_s^{x,{w}}, b(X_s^{x,\widetilde{w}})- b(X_s^{x,{w}})\rangle\right) ds.
   \end{align*}
   By Assumption $\HUN$,
   $$\langle X_s^{x,\widetilde{w}}-X_s^{x,{w}}, b(X_s^{x,\widetilde{w}})- b(X_s^{x,{w}})\rangle\le -\alpha |X_s^{x,\widetilde{w}}-X_s^{y,w}|^2$$
   whereas by the elementary inequality $|uv|\le (\varepsilon/2)|u|^2+1/(2\varepsilon)|v|^2$ applied with $\varepsilon=\alpha$,
   $$2\langle X_s^{x,\widetilde{w}}-X_s^{x,{w}}), \sigma D'_s (\widetilde{w}-w)\rangle\le \alpha |X_s^{x,\widetilde{w}}-X_s^{y,w}|^2+
   \frac{\|\sigma\|^2}{\alpha}|D'_s (\widetilde{w}-w)|^2.$$
   Thus, for any $\varepsilon'>0$, we have for any $t\ge t_0$,
   $$e^{\alpha t} |X_t^{x,\widetilde{w}}-X_t^{y,w}|^2\le e^{\alpha t_0}|X_{t_0}^{x,\widetilde{w}}-X_{t_0}^{y,w}|^2+
   \frac{\|\sigma\|^2}{\alpha}\int_{t_0}^t e^{\alpha s}|D'_s (\widetilde{w}-w)|^2ds.$$
   By Lemma \ref{propDmoins}, we deduce that a positive constant $C$  exists such that:
   $$\int_{t_0}^t e^{\alpha s}|D'_s (\widetilde{w}-w)|^2ds\le C\Phi_{\varepsilon,\varepsilon'}(\tilde{w}-w)\left(1+ \int_1^{t\vee 1} e^{\alpha s}s^{2H-2+2\varepsilon'} ds\right).$$
   By an integration by parts, it follows that
   \begin{align*}
   \frac{e^{-\alpha t}}{\alpha} \int_{t_0}^t e^{\alpha s}|D'_s (\widetilde{w}-w)|^2ds&\le
   \frac{C\Phi_{\varepsilon,\varepsilon'}(\tilde{w}-w)}{\alpha}\left(e^{-\alpha t} +\int_1^{t\vee 1} e^{\alpha (s-t)}s^{2H-2+2\varepsilon'} ds\right)\\
   &\le C\Phi_{\varepsilon,\varepsilon'}(\tilde{w}-w))
   (t\vee 1)^{2H-2+2\varepsilon'}.
   \end{align*}
   Thus,
   $$|X_t^{x,\widetilde{w}}-X_t^{y,w}|^2\le e^{-\alpha (t-t_0)}|X_{t_0}^{x,\widetilde{w}}-X_{t_0}^{y,w}|^2+C\|\sigma\|^2 {\Phi_{\varepsilon,\varepsilon'}(\tilde{w}-w)}(t\vee 1)^{2H-2+2\varepsilon'}.$$
   Let us finally control $|X_{t_0}^{x,\widetilde{w}}-X_{t_0}^{y,w}|^2$. Since $b$ is $L$-Lipschitz continuous, for every $t\ge 0$,
   $$|X_{t}^{x,\widetilde{w}}-X_{t}^{y,w}|\le |x-y|+L \int_0^t |X_{s}^{x,\widetilde{w}}-X_{s}^{y,w}| ds+\|\sigma\|\sup_{s\in(0,t]} |D_s(\widetilde{w}-w)|,$$
   and the Gronwall Lemma yields:
   $$|X_{t_0}^{x,\widetilde{w}}-X_{t_0}^{y,w}|\le \left(|x-y|+\|\sigma\|\sup_{s\in(0,t_0]} |D_s(\widetilde{w}-w)|\right)e^{L t_0}.$$
   Now, assume that $t_0\in(0,1]$. By Lemma \ref{propDmoins},
   $$\sup_{s\in(0,t_0]} |D_s(\widetilde{w}-w)|\le \int_0^{t_0}|D'_s(\widetilde{w}-w)| ds\le C \Phi_{\varepsilon,\varepsilon'}(\tilde{w}-w) t_0^{H-\varepsilon}.$$
   As a consequence,
   $$|X_{t_0}^{x,\widetilde{w}}-X_{t_0}^{y,w}|^2\le C\left(|x-y|^2 +\Phi_{\varepsilon,\varepsilon'}(\tilde{w}-w)\right).$$
   The result follows.
   \end{proof}
   Before stating Proposition \ref{prop:controlfo} (which provides the exponential bound for $\e[F_X~|~\mathcal{F}_0]-\e[F_X]$), we finally obtain bounds on  the moments of $\Phi_{\varepsilon,\varepsilon'}(W^-)$ and of the invariant distribution.

   \begin{Lemma} \label{lem:contbrownnorm} Let $\varepsilon, \varepsilon'$ be some positive numbers. Then, there exists {$C_{\varepsilon'}>0$} such that for every $\lambda>0$,
   $$\e[|\Phi_{\varepsilon,\varepsilon'}(W^-)|^p]\le C_{\varepsilon'}^p \e[|{\cal Z}|^p]$$
   where ${\cal Z}$ has ${\cal N}(0,1)$-distribution.
   \end{Lemma}
   \begin{proof} It is enough to consider the one-dimensional case and by  a symmetry argument, it is certainly equivalent to prove the result for a Brownian motion on $\ER_+$. Furthermore, using that for a Brownian motion $W$ on $\ER_+$,
   $t\mapsto t W_{\frac{1}{t}}$ (with initial value equal to $0$) is also a Brownian motion, we deduce that
   we have only to prove that for any $\varepsilon'>0$,
   $$\e\left[\sup_{t\ge1}\left|\frac{W_t}{t^{\frac{1}{2}+\varepsilon'}}\right|^p\right]\le C_{\varepsilon'}^p \e[|{\cal Z}|^p].$$
   By the It\^o formula,
   \begin{equation}\label{eq:decompwsur}
   \frac{W_t}{t^{\frac{1}{2}+\varepsilon'}}=W_1+\int_1^{t} \frac{1}{s^{\frac{1}{2}+\varepsilon'}}dW_s-(\frac{1}{2}+\varepsilon')\int_1^{t} \frac{W_s}{s^{\frac{3}{2}+\varepsilon'}} ds.
   \end{equation}
   For the first right-hand side term, there is nothing to prove. For the second one, we remark that it is a Gaussian process  and it follows that a Brownian Motion $\tilde{W}$ exists such that
   $$\sup_{t\ge1} \left|\int_1^{t} \frac{1}{s^{\frac{1}{2}+\varepsilon'}}dW_s\right|\overset{(d)}{=}\sup_{t\in[0,\sigma_\infty)} |\tilde{W}_t|\quad\textnormal{where}\quad \sigma_\infty=\int_1^{+\infty} \frac{1}{s^{{1+2\varepsilon'}}} ds<+\infty, $$
   and ``$\overset{(d)}{=}$'' stands for the equality in distribution. But
   $$\sup_{t\in[0,\sigma_\infty)} |\tilde{W}_t|\le \sup_{t\in[0,\sigma_\infty)} \tilde{W}_t-\sup_{t\in[0,\sigma_\infty)} (-\tilde{W}_t)$$
   and hence
   $$\e[\sup_{t\in[0,\sigma_\infty)} |\tilde{W}_t|^p]\le 2^p \e[|\sup_{t\in[0,\sigma_\infty)} \tilde{W}_t|^p]\le (2\sqrt{\sigma_\infty})^{p} \e[|{\cal Z}|^p]$$
   since $\sup_{t\in[0,\sigma_\infty)} \tilde{W}_t$ has the same distribution as $\sqrt{\sigma_\infty} |{\cal Z}|$ where ${\cal Z}$ has ${\cal N}(0,1)$-distribution.
   Let us now consider the last term of \eqref{eq:decompwsur}. We have
   $$\sup_{t\ge1}\left|\int_1^{t} \frac{W_s}{s^{\frac{3}{2}+\varepsilon'}} ds\right|\le \int_1^{+\infty} \frac{|W_s|}{s^{\frac{3}{2}+\varepsilon'}} ds.$$
   Hence, by the Jensen inequality applied with the probability measure $\mu_{\varepsilon'}(ds)=Cs^{-(1+\varepsilon')} ds$, we get:
   $$\e\left[\sup_{t\ge1}\left|\int_1^{t} \frac{W_s}{s^{\frac{3}{2}+\varepsilon'}} ds\right|^p\right]\le
   C \int_1^{+\infty} \e\left|\frac{W_s}{\sqrt{s}}\right|^p \frac{1}{s^{1+\varepsilon'}} ds.$$
   Thus, by the scaling property, it follows that
   $$\e\left[\sup_{t\ge1}\left|\int_1^{t} \frac{W_s}{s^{\frac{3}{2}+\varepsilon'}} ds\right|^p\right]\le C\e[|{\cal Z}|^p]$$
   where ${\cal Z}$ has ${\cal N}(0,1)$-distribution.
   \end{proof}
   \begin{Lemma}\label{contmomentinvar} Assume $\HUN$ and $\HDEUX$ and let $\bar{\nu}$ denote the \textit{marginal} invariant distribution. Then, {there exists $C=C(H,L,\alpha,|b(0_{\R^d})|,\|\sigma\|)>0$ such that} for any $p>1$,
   $$ \int_{\R^d} |y|^p \bar{\nu}(dy)\le C^p\e[|{\cal Z}|^p]$$
   where ${\cal Z}$ denotes a random variable with ${\cal N}(0,1)$-distribution.{Furthermore, $(L,\alpha,b_0,s)\mapsto C(H,L,\alpha,b_0,s)$ is bounded on every compact set of $\ER_+\times\ER_+^*\times\ER_+\times\ER_+$.}
   \end{Lemma}

   \begin{Remark} The dependency of $C$ with respect to $H,L,\alpha,b(0_{\R^d})$ and $\sigma$ is explicit and is given in the following proof.
   \end{Remark}

   \begin{proof} The proof uses some arguments of \cite[Proposition 3.12]{hairer} by controlling the distance between the solution to the SDE with the one of a Ornstein-Uhlenbeck process for which the announced property holds. For the sake of completeness, let us give some details. Let $(U_t)_{t\ge0}$ denote a solution to $dU_t=-U_t dt+\sigma dB_t^H$ and $(X_t)_{t\ge0}$ a solution to  \eqref{modeldiffusion}. Assume that $U$ and $X$ are built with the same fBm and start from the same starting point $x$. Then,
   $$ X_t-U_t=\int_0^{t} b(X_s)-U_s ds$$
   so that
   $$ |X_t-U_t|^2=2\int_0^t \langle b(X_s)-U_s, X_s-U_s\rangle ds.$$
   For any $x$ and $u\in\ER^d$,
   $$\langle b(x)-u, x-u\rangle=\langle b(x)-b(u), x-u\rangle+\langle b(u)-u, x-u\rangle.$$
   Now, by $\HUN$ and $\HDEUX$ (which implies that {$|b(u)|\le (|b(0_{\R^d})|\vee L)(1+|u|)$}), we get
   $$\langle b(x)-u, x-u\rangle\le -\frac{\alpha}{2}|x-u|^2+{\frac{(|b(0_{\R^d})|\vee L+1)^2}{\alpha}(1+|u|^2)}.$$
   Thus, with similar arguments as in the proof of Lemma \ref{lem:conttrajxw} (based on Gronwall-type arguments), we deduce that for any $t\ge 0$,
   $$ |X_t-U_t|^2\le {\frac{2(|b(0_{\R^d})|\vee L+1)^2}{\alpha}}\int_0^te^{{\alpha}(s-t)}{(1+|U_s|^2 )}ds.$$
   Thus, denoting by $\|.\|_p$, the $L^p$-norm we deduce from Jensen inequality that
   $$\|X_t^x\|_p^{{p}}\le {2^{p-1}\left(\frac{4(|b(0_{\R^d})|\vee L+1)^2}{\alpha}\times\left(\frac{1}{\alpha}\vee1\right)\right)^{p/2}} \int_0^te^{{\alpha}(s-t)}(1+\|U_s^x\|_p^{{p}} )ds+{2^{p-1}}\|U_t^x\|_p^{{p}}.$$
   Denote by $\bar{\nu}$ and $\pi_\sigma$ the (marginal) invariant distributions of $X$ and $U$. {Owing to uniform integrability arguments and to  the convergence in distribution of $(X_t)_{t\ge0}$ and $(U_t)_{t\ge0}$ towards $\bar{\nu}$ and $\pi_\sigma$, we get:}
   $$\int_{\R^d}|y|^p\bar{\nu}(dy)=\lim_{t\rightarrow+\infty}\|X_t^x\|_p^p\le {\bar{C}}^p\lim_{t\rightarrow+\infty}\|U_t^x\|_p^p
   ={\bar{C}}^p\int_{\R^d}|y|^p\pi_\sigma(dy)$$
   {where $$\bar{C}:=2\left(\frac{4(|b(0_{\R^d})|\vee L+1)^2}{\alpha}\times\left(\frac{1}{\alpha}\vee1\right)\right)^{1/2}+2.$$}
   {Finally, let us recall that by a standard integration by parts,
   $$ U_t=x e^{-t}+\sigma \int_0^t e^{s-t} d B_s^H$$
   and it follows that $\pi_\sigma=\pi_{I_d}\circ \varphi_\sigma^{-1}$ where $\varphi_\sigma=\sigma x.$ Thus,
   $$\int_{\R^d}|y|^p\pi_\sigma(dy)\le \|\sigma\|^p \int |y|^p\pi_{I_d}(dy).$$
But by \cite[Proposition 3.12]{hairer}, $\pi_{Id}$ has Gaussian distribution ${\cal N}(0_{\ER^d},c_0 I_d)$ where $c_0\le \Gamma(2H+1)$, so that the result follows with ${C}=\|\sigma\| (d\Gamma(2H+1))^{\frac{1}{2}}\bar{C}$ (which has the local boundedness property announced in the lemma).}
   \end{proof}

   \noindent We are now in position to provide an exponential bound for $\e[F_X~|~\mathcal{F}_0]-\e[F_X]$:
   \begin{Proposition}\label{prop:controlfo} Assume $\HUN$ and $\HDEUX$. Suppose that $n\Delta_n\ge 1$. Then, for any $\varepsilon'\in(0,1-H)$, a constant {$C=C(H,L,\alpha,b(0_{\R^d}),\sigma,\varepsilon')$} exists such that
   $$\e[\exp\left(\lambda (\e[F_X~|~\mathcal{F}_0]-\e[F_X])\right)]\le \exp\left(\lambda^2 u_0^{(n)}\right),
   $$
   where
   $$u_0^{(n)}=C\|F\|^2_{\rm Lip}\frac{(n\Delta_n)^{H+\varepsilon'}}{\Delta_n}.$$
   In particular, with $\varepsilon'=\frac{1}{2}$ and $\varepsilon'=\frac{1-H}{2}$ when $H<1/2$ and $H>1/2$ respectively,
   $$ u_0^{(n)}\le C\|F\|^2_{\rm Lip}
   \begin{cases}
   n\Delta_n^{-1}&\textnormal{if $H<1/2$,}\\
   n^{H+\frac{1-H}{2}}\Delta_n^{\frac{H-1}{2}}&\textnormal{if $H>1/2$.}
   \end{cases}$$
   \end{Proposition}
   \begin{Remark} \label{rem:constant2}{Following carefully the constants involved in the proof below (induced by the previous lemmas), one checks that for every $H\in(0,1)$ and $\varepsilon'\in(0,1-H)$, $(L,\alpha,b_0,s)\mapsto C(H,L,\alpha,b_0,s,\varepsilon')$ is bounded on every compact set of $\ER_+\times\ER_+^*\times\ER_+\times\ER_+$. Thus, since the proof of Theorem \ref{main_thm_concentration} is obtained as a combination of Propositions \ref{prop:exp_moments_sum_martingale_inc} and \ref{prop:controlfo}, this property combined with Remark \ref{rem:constant1} implies that the constant  $\constPV$  of Theorem \ref{main_thm_concentration} has the local boundedness property announced in Remark \ref{rem:theoconcen}.}\end{Remark}

   \begin{Remark} The above bound easily involves that the contribution of $\e[F_X~|~\mathcal{F}_0]-\e[F_X]$ is always less constraining (up to a multiplicative constant) than the one obtained in Proposition \ref{prop:exp_moments_sum_martingale_inc}.
   \end{Remark}
   \begin{proof} By \eqref{eq:controlfxfo} and Lemma \ref{lem:conttrajxw}, for any $p>1$,
   \begin{align*}
   |\e[F_X~|~\mathcal{F}_0&]-\e[F_X]|^p\le (2C_1\|F\|_{\rm Lip})^p\left( \int_{\R^d} |X_0-y|\bar{\nu}(dy)\sum_{k=1}^n e^{-\frac{\alpha }{2} t_k}\right)^p\\
   &+(2C_2\|F\|_{\rm Lip})^p\left(\int\left|\Phi_{\varepsilon,\varepsilon'}(W^{-}-w)\right|^p\mathproba{P}_{W^-}(dw)\right)\left(\sum_{k=1}^n (t_k\vee 1)^{H-1+\varepsilon'}\right)^p
   \end{align*}
   where $\bar{\nu}$ stands for the ``marginal'' invariant distribution, $i.e.$ the projection of $\nu$ on the first coordinate and $\tilde{w}$ denotes the p
   Let us consider the two right-hand side terms separately. On the one hand, using Jensen inequality, we get
   $$\left(\int_{\R^d} |X_0-y|\bar{\nu}(dy)\sum_{k=1}^n e^{-\frac{\alpha }{2} t_k}\right)^p\le \int_{\R^d} |X_0-y|^p\bar{\nu}(dy)\left(\frac{2}{\alpha \Delta_n}\right)^p.$$
   Thus, using that $|X_0-y|^p\le 2^p\left(|X_0|^p+|y|^p\right)$ and Lemma \ref{contmomentinvar}, we get
   \begin{equation}\label{eq:123212}
   \e \left(\int_{\R^d} |X_0-y|\bar{\nu}(dy)\sum_{k=1}^n e^{-\frac{\alpha }{2} t_k}\right)^p\le
   \left(\frac{4}{\alpha \Delta_n}\right)^p C^p \e[|{\cal Z}|^p].
   \end{equation}
   On the other hand, since $H-1+\varepsilon'>-1$,
   $$\sum_{k=1}^n (t_k\vee 1)^{H-1+\varepsilon'}\le C\left(\frac{1}{\Delta_n}+{n^{H+\varepsilon'}}{\Delta_n}^{H-1+\varepsilon'}\right)\le  C\frac{(n\Delta_n)^{H+\varepsilon}}{\Delta_n}. $$
   Furthermore,
   $$\e\int_{\Omega_-}\left|\Phi_{\varepsilon,\varepsilon'}(W^--w)\right|^p\mathproba{P}_{W^-}(dw)\le 2^p \e|\Phi_{\varepsilon,\varepsilon'}(W^-)|^p.$$
   Thus, by Lemma \ref{lem:contbrownnorm}, it follows that
   \begin{equation}\label{eq:12321}
   \e\int_{\Omega_-}\left|\Phi_{\varepsilon,\varepsilon'}(W^--w)\right|^p\mathproba{P}_{W^-}(dw)\left(\sum_{k=1}^n (t_k\vee 1)^{H-1+\varepsilon'}\right)^p\le
   \left(C\frac{(n\Delta_n)^{H+\varepsilon}}{\Delta_n}\right)^p \e[|{\cal Z}|^p].
   \end{equation}
   Combining \eqref{eq:123212}, \eqref{eq:12321} and the fact that
   $\e|{\cal Z}|^p\le C^p \Gamma\left(\frac{p+1}{2}\right)\le C^p p\Gamma\left(\frac{p}{2}\right)$,
   we get: there exists a constant $C>0$ such that for all $p\ge 2$,
   $$\e|\e[F_X~|~\mathcal{F}_0]-\e[F_X]|^p\le \left(C\|F\|_{\rm Lip}\frac{(n\Delta_n)^{H+\varepsilon}}{\Delta_n}\right)^p p\Gamma\left(\frac{p}{2}\right).
   $$
   To conclude, we apply Lemma \ref{lem:subgaussian}.
   \end{proof}

   \section{Proofs of Statistical properties}\label{sec:proofs2}

   \subsection{Proof of Proposition~\ref{prop:bias}}

   \paragraph{Step 1.}
   Below we denote $\x=(t_1,\dotsc,t_d)$ and we define for $h\in\cH$ and $\eta$  such that $\eta_i\in\{0, h_i\}$:
   \begin{equation}
      v_i(u) = (t_1-\eta_1u_1,\ldots,t_{i-1}-\eta_{i-1}u_{i-1},
      \quad t_i\quad,
      t_{i+1}-h_{i+1}u_{i+1}, \ldots, t_d-h_du_d).
   \end{equation}
   We can write:
   \begin{align}
      f(\x-h\cdot u)-f(\x-\eta\cdot u)
      &= \sum_{i=1}^d f(v_i(u)-h_iu_ie_i) - f(v_i(u)-\eta_iu_ie_i)\\
      &= \sum_{i\in I} f(v_i(u)-h_iu_ie_i) - f(v_i(u)),
   \end{align}
   where $I=\{i=1,\ldots,d : \eta_i=0\}$. Now fix $i\in I$.

   If $\lfloor s_i\rfloor = 0$ then we obtain:
   \begin{equation*}
      |f(v_i(u)-h_iu_ie_i) - f(v_i(u)) | \leq L_i |h_iu_i|^{s_i}
   \end{equation*}
   which leads to (since $f\in \Sigma_d(\boldsymbol{s},\bold L)$)
   \begin{equation*}
      \left|\int_\R K(u_i) \left( f(v_i(u)-h_iu_ie_i) - f(v_i(u)) \right) d u_i\right|
      \leq L_i h_i^{s_i} \int_\R \big|u_i^{s_i} K(u_i)\big| du_i.
   \end{equation*}
   Otherwise, using a Taylor expansion of the function $z\in\R\mapsto f(v_i(u)+ze_i)$ around $0$, we obtain:
   \begin{align}
      f(v_i(u)-h_iu_ie_i) - f(v_i(u))
      &=
         {\sum_{k=1}^{\lfloor s_i\rfloor}}
         D_i^k f(v_i(u))\frac{(-h_iu_i)^k}{k!}\\
      &+\frac{(-h_iu_i)^{\lfloor s_i\rfloor}}{\lfloor s_i\rfloor!}
         \left[\frac{\partial^{\lfloor s_i\rfloor}f}{\partial x_i^{\lfloor s_i\rfloor}} (v_i(u)-\tau h_iu_i)-\frac{\partial^{\lfloor s_i\rfloor}f}{\partial x_i^{\lfloor s_i\rfloor}}(v_i(u))\right],
   \end{align}
   where $\tau\in(0,1)$. This implies that, using that $K$ is a kernel of order larger than $\lfloor s_i \rfloor$ combined with the fact that $v_i(u)$ does not depend on $u_i$,
   \begin{equation*}
      \left|\int_\R K(u_i) \left( f(v_i(u)-h_iu_ie_i) - f(v_i(u)) \right) d u_i\right|
      \leq \frac{L_i h_i^{s_i}}{\lfloor s_i \rfloor !} \int_\R \big|u_i^{s_i} K(u_i)\big| du_i.
   \end{equation*}

   Combining the above results we obtain:
   \begin{equation*}
      \left|\int_{\R^d} \prod_{i=1}^d K(u_i) \left(f(\x-h\cdot u)-f(\x-\eta\cdot u)\right) du \right|
      \leq \sum_{i\in I} \frac{L_i h_i^{s_i}}{\lfloor s_i \rfloor !} \int_\R \big|u_i^{s_i} K(u_i)\big| du_i.
   \end{equation*}

   \paragraph{Step 2.}
   Let $h, \kh \in\cH$. Taking $\eta=0$ in step 1 we obtain \eqref{eq:bias}
   Taking $\eta_i = 0$ if $h_i = \kh_i \vee h_i$ and $\eta_i = \kh_i = h_i\vee \kh_i$ otherwise, we obtain:
   \begin{align*}
      \left|\e \hat{f}_{h\vee\kh}(x_0)-\e\hat{f}_{\kh}(x_0)\right|
      &\leq
      \left|\e \hat{f}_{h\vee\kh}(x_0)-\e\hat{f}_{\eta}(x_0)\right|
      + \left|\e \hat{f}_{\eta}(x_0)-\e\hat{f}_{\kh}(x_0)\right|\\
      &\leq 2\sum_{i\in I} \frac{L_i h_i^{s_i}}{\lfloor s_i \rfloor !} \int_\R \big|u_i^{s_i} K(u_i)\big| du_i\\
      &\leq 2\sum_{i=1}^d \frac{L_i h_i^{s_i}}{\lfloor s_i \rfloor !} \int_\R \big|u_i^{s_i} K(u_i)\big| du_i.
   \end{align*}
   This implies \eqref{eq:bias2}.

   \subsection{Proof of Proposition~\ref{prop:stoc}}
   We have
   \begin{equation}
      \e\left|\hat f_h(\xx) - \e\hat f_h(\xx)\right|^p
      =\int_0^\infty \p\left(\left|\hat f_h(\xx) - \e\hat f_h(\xx)\right|^p>t\right)dt.
   \end{equation}
   We now use Corollary~\ref{cor:concentration_inequality} with the functional $g(u)=K_h (x_0-u)$. We obtain that
   \begin{align}
      \e\left|\hat f_h(\xx) - \e\hat f_h(\xx)\right|^p & = p\int_0^\infty u^{p-1} \p\left(\left|\hat f_h(\xx) - \e\hat f_h(\xx)\right| >u\right)du\\
      &\le  2p\int_0^\infty u^{p-1} \exp\left(-\frac{u^2 (n\Delta)^{\betaH}}{4\constPV\|g\|^2_{\Lip}}\right)du\\
      &=  2p\left(\frac{4\constPV\|g\|_{\Lip}^2}{(n\Delta_n)^{\betaH}}\right)^{p/2} \int_0^{+\infty} u^{p-1} \exp\left( -u^2 \right) du\\
      &= p\Gamma\left(\frac{p+1}{2}\right) \left(\frac{4\constPV\|g\|_{\Lip}^2}{(n\Delta_n)^{\betaH}}\right)^{p/2}\\
      &= p\Gamma\left(\frac{p+1}{2}\right)
      \constPV^{p/2}
      \left(\frac{4\|g\|_{\Lip}^2}{(n\Delta_n)^{\betaH}}\right)^{p/2}
   \end{align}

   Since $\|g\|_{\Lip} \leq \sqrt{d}\LK V_h^{-1}$, we obtain \eqref{eq:stoc}.

   \subsection{Proof of oracle inequality}

   We split the proof of Theorem~\ref{thm:oracle} into several steps.\\

   \paragraph{Step 1.} Let $h\in\cH$ be an arbitrary bandwidth. Using triangular inequality we have:
   \begin{equation}
      |\hat{f}(x_0)-f(x_0)|\le |\hat{f}_{\hat{h}}(x_0)-\hat{f}_{h\vee\hat{h}}(x_0)|+|\hat{f}_{h\vee\hat{h}}(x_0)-\hat{f}_{h}(x_0)|+|\hat{f}_{h}(x_0)-f(x_0)|.
   \end{equation}
   Note that
   \begin{align}
      |\hat{f}_{\hat{h}}(x_0)-\hat{f}_{h\vee\hat{h}}(x_0)|
      &\leq \left\{|\hat{f}_{\hat{h}}(x_0)-\hat{f}_{h\vee\hat{h}}(x_0)|- M_n(h,\hat h)\right\}_+ + M_n(h,\hat h)\\
      &\leq \max_{\kh\in\cH}\left\{|\hat{f}_{\kh}(x_0)-\hat{f}_{h\vee\kh}(x_0)|-M_n(h,\kh)\right\}_+ + M_n(h,\hat h)\\
      &\leq B(h,\x) + M_n(\hat h) + M_n(h).
   \end{align}
   Applying the same reasoning to the term $|\hat{f}_{h\vee\hat{h}}(x_0)-\hat{f}_{h}(x_0)|$ and using \eqref{eq:selection-rule}, this leads to
   \begin{align*}
      |\hat{f}(x_0)-f(x_0)|
      &\leq B(h,x_0)+ 2 M_n(h)
      + B(\hat h,x_0) + 2 M_n(\hat h)
      +|\hat{f}_{h}(x_0)-f(x_0)|\\
      &\leq 3 B(h,x_0)+ 4 M_n(h)
      +|\hat{f}_{h}(x_0)-f(x_0)|
   \end{align*}
   This implies that:
   \[
   R_n(\hat{f},f) \leq
   3 \left(\e B^p(h, \x)\right)^{1/p} +
   4 M_n(h) +
   R_n(\hat f_h, f)
   \]

   \paragraph{Step 2.}
   Now, we upper bound $B(h,\x)$. Using basic inequalities we have:
   \begin{align}
      B(h,\x)
      &\leq
      \max_{\kh\in\cH}
      \left\{
      \left|\e \hat{f}_{h\vee\kh}(x_0)-\e\hat{f}_{\kh}(x_0)\right|
      \right\}_+ +
      \max_{\kh\in\cH}
      \left\{
      \left|\hat{f}_{\kh}(x_0)-\e\hat{f}_{\kh}(x_0)\right|
      -{M}_n(\kh)\right\}_+\\
      &\qquad +
      \max_{\kh\in\cH}
      \left\{
      \left|\hat{f}_{h\vee\kh}(x_0)-\e \hat{f}_{h\vee\kh}(x_0)\right|
      -{M}_n(h\vee\kh)\right\}_+  \\
      &\leq E_h(\x) + 2 T,
   \end{align}
   where
   \begin{equation}
      T=\max_{\kh\in\cH}
      \left\{
      \left|\hat{f}_{\kh}(x_0)-\e\hat{f}_{\kh}(x_0)\right|
      -{M}_n(\kh)\right\}_+.
   \end{equation}

   This leads to:
   \begin{equation}\label{eq:step3}
      \left(\e B^p(h,\x)\right)^{1/p} \leq E_h(\x)+2\left(\e T^p\right)^{1/p}.
   \end{equation}

   \paragraph{Step 3.}

   We have:
   \begin{align*}
      \e T^p
      &\le \sum_{\kh\in\cH}\int_0^{\infty} \p
      \left(\left|\frac1n \sum_{i=1}^n \bar g_\kh(X_{t_i})\right|\ge M_n(\kh)+t^{1/p}\right)dt,
   \end{align*}
   where
   \[
   g_\kh(X_{t_i}) = K_\kh(\x-X_{t_i})
   \qquad\text{and}\qquad
   \bar g_\kh(X_{t_i}) = g_\kh(X_{t_i}) - \e g_\kh(X_{t_i}).
   \]
   We obtain, using Corollary~\ref{cor:concentration_inequality}:
   \begin{align*}
      \e T^p
      &\leq p\sum_{\kh\in\cH}\int_0^{\infty} u^{p-1}\p
      \left(\left|\frac1n \sum_{i=1}^n \bar g_\kh(X_{t_i})\right|\ge M_n(\kh)+u\right)du \\
      &\leq  2p\sum_{\kh\in\cH}\int_0^\infty
      u^{p-1}\exp\left(-\frac{(u+M_n(\kh))^2 (n\Delta_n)^{\betaH}}{4\constPV\|g_\kh\|^2_{\Lip}}\right)du\\
      &\leq 2p\sum_{\kh\in\cH}\int_0^\infty
      u^{p-1}\exp\left(-\frac1{\constPV}\left(\frac{u+M_n(\kh)}{\varphi_n(\kh)}\right)^2\right)du\\
      &\leq 2p\sum_{\kh\in\cH}\int_0^\infty
      u^{p-1}\exp\left(-{\frac1{\constPV}}\left(\frac{u}{\varphi_n(\kh)}\right)^2\right)
      \exp\left( -p\frac{\kkk}{\constPV}|\log V_\kh|  \right)
      du
   \end{align*}
   Since $\kkk>\constPV$ we obtain:
   \begin{align*}
      \e T^p&\leq p\Gamma\left( \frac{p+1}{2} \right) \sum_{\kh\in\cH} \left(\constPV^{1/2}\varphi_n(\kh) V_\kh\right)^{p}
     {V_\kh^{p\left( \frac{\kkk}{\constPV}-1 \right)}}\\
      &\leq p\Gamma\left( \frac{p+1}{2} \right) (4d\LK^2)^{p/2}
      \left(\frac{\constPV}{(n\Delta_n)^{\beta{H}}}\right)^{p/2}{\sum_{\kh\in\cH}V_\kh^{p\left( \frac{\kkk}{\constPV}-1 \right)}}\\
      &\leq  p{\left( \frac{e^{p\left( \frac{\kkk}{\constPV}-1 \right)}}{e^{p\left( \frac{\kkk}{\constPV}-1 \right)}-1} \right)^d}\Gamma\left( \frac{p+1}{2} \right) (4d\LK^2)^{p/2}
      \left(\frac{\constPV}{(n\Delta_n)^{\beta{H}}}\right)^{p/2}
   \end{align*}
   Finally we obtain the following upper bound:
   \begin{equation*}
      \left( \e T^p \right)^{1/p}
      \leq 2\LK\sqrt{d} \left( p\Gamma\left( \frac{p+1}{2} \right)
      {\left(
      \frac{e^{p\left( \frac{\kkk}{\constPV}-1 \right)}}{e^{p\left( \frac{\kkk}{\constPV}-1 \right)}-1}
      \right)^d}
      \right)^{1/p}
      \times \left(\frac{\constPV}{(n\Delta_n)^{\beta{H}}}\right)^{1/2}.
   \end{equation*}
   This allows us to obtain Theorem~\ref{thm:oracle}.

   \subsection{Proof of Theorem~\ref{thm:adaptive}}

   Set $\mathbf{s}\in(0,M+1]^d$, $\mathbf{L}\in(0, +\infty)^d$.
   {To prove this result, we construct a specific bandwidth vector $h^*$ that belongs to $\cH$. This allows to apply Propositions~\ref{prop:bias} and~\ref{prop:stoc} and to bound, in~\eqref{eq:thm-oracle}, the minimum over $h\in\cH$ by the value for $h=h^*$}.
   Let
   \begin{equation*}
      l_i^* = {\left\lfloor \frac{\gamma(\mathbf{s})}{s_i}
      \left(
         \log\left( (n\Delta_n)^{\betaH} \right)
         -
         \log \log\left((n\Delta_n)^{\betaH} \right)
      \right)
      \right\rfloor}.
   \end{equation*}
   Since $\gamma(\mathbf{s})/s_i \leq 1/2$ we have
   \begin{equation}\label{eq:th4-1}
      0\leq l_i^* \leq  \frac12
      \log\left( {(n\Delta_n)^{\betaH}} \right)
      \leq \frac{\betaH}2
      \log\left( n\Delta_n \right).
   \end{equation}
   Now, denote for $i=1,\dotsc, d$:
   \begin{equation*}
      h_i^* = e^{-l_i^*}
      \quad\text{and}\quad
      h_i(\mathbf{s}) = \left({(n\Delta_n)^{-\betaH}}\log\left({(n\Delta_n)^{\betaH}}\right)\right)^{\gamma(\mathbf{s})/s_i}.
   \end{equation*}
   Remark that, using these notations we have $h_i(\mathbf{s}) \leq h_i^* \leq e h_i(\mathbf{s})$.  If we consider $h^*=(h^*_1,\dotsc, h^*_d)$ and $h(\mathrm s) = (h_1(\mathrm s), \dotsc, h_d(\mathrm s))$, then:
   \begin{align*}
      V_{h^*}
      &\geq V_{h(\mathbf{s})} \\
      &\geq
      \left((n\Delta_n)^{-\betaH}\log\left({(n\Delta_n)^{\betaH}}\right)\right)^{\gamma(\mathbf{s})(1/\bar s + 1/s_\mathrm{min})}\\
      &\geq
      \left((n\Delta_n)^{-\betaH}\right)^{\gamma(\mathbf{s})(1/\bar s + 1/s_\mathrm{min})}\\
      &= \left((n\Delta_n)^{-\betaH}\right)^{1/2-\gamma(\mathbf{s})}\\
      &\geq (n\Delta_n)^{-\betaH/2} ,
   \end{align*}
   where $s_\mathrm{min} = \min_j s_j$ and using that $(n\Delta_n)^{\beta_H}\ge e$. This implies, in combination with~\eqref{eq:th4-1}, that $h^*\in\cH$. In~\eqref{eq:thm-oracle} we can bound the right hand side by taking $h=h^*$. Let us consider each term separately.

   First, using Propositions~\ref{prop:bias} and~\ref{prop:stoc}, since $0<s_i\leq M+1$ for each $i$, we have
   \begin{align}
      R_n(\hat{f}_{h^*}, f)
      \leq &
      \left(e^{M+1}\Lambda_1 + \frac{\Lambda_2}{(\log((n\Delta_n)^\betaH))^{1/2}}\right)
   \left(\frac{\log((n\Delta_n)^\betaH)}{(n\Delta_n)^{\betaH}}\right)^{\gamma(\boldsymbol{s})}\\
   \leq & \left(e^{M+1}\Lambda_1 + \Lambda_2\right)
   \left(\frac{\log((n\Delta_n)^\betaH)}{(n\Delta_n)^{\betaH}}\right)^{\gamma(\boldsymbol{s})} .
   \end{align}
   Secondly, using Proposition~\ref{prop:bias} we obtain:
   \begin{equation*}
      E_{h^*}(\xx) \leq
      2e^{M+1}\Lambda_1
      \left(\frac{\log((n\Delta_n)^\betaH)}{(n\Delta_n)^{\betaH}}\right)^{\gamma(\boldsymbol{s})}.
   \end{equation*}
   Finally, using~\eqref{eq:Vhvarphi} and~\eqref{eq:MNHH} we have:
   \begin{align*}
      M_n(h^*) &= \left(\frac{4d\LK^2}{V_{h^*}^2(n\Delta_n)^{\betaH}}\right)^{1/2}\sqrt{\kkk p |\log V_{h^*}|}\\
      &\leq 2\sqrt{\kkk d p}\LK   \left(\frac{1}{V_{h(\mathbf{s})}^2(n\Delta_n)^{\betaH}}\right)^{1/2}   \left( |\log (e^{d+1}V_{h(\mathbf{s})})| \right)^{1/2}\\
      & \leq 2\sqrt{\kkk d p}\LK \left(\frac{\log((n\Delta_n)^\betaH)}{(n\Delta_n)^{\betaH}}\right)^{\gamma(\boldsymbol{s})}  \left( \frac{|\log (e^{d+1}V_{h(\mathbf{s})})|}{ \log((n\Delta_n)^\betaH) } \right)^{1/2}\\
       & \leq 2\sqrt{\kkk d p}\LK \left(\frac{\log((n\Delta_n)^\betaH)}{(n\Delta_n)^{\betaH}}\right)^{\gamma(\boldsymbol{s})}  \left( \frac{d+1+(1/2-\gamma(\boldsymbol{s}))  \log((n\Delta_n)^\betaH) }{ \log((n\Delta_n)^\betaH) } \right)^{1/2}\\
       & \leq 2\sqrt{\kkk d(d+3/2) p}\LK \left(\frac{\log((n\Delta_n)^\betaH)}{(n\Delta_n)^{\betaH}}\right)^{\gamma(\boldsymbol{s})}
   \end{align*}
   This allows us to conclude that $\min_{h \in\cH} \left\{
            R_n(\hat{f}_h, f) +
            4 M_n(h)
            + 3 E_h(\x)
         \right\}$ is bounded up to a multiplicative constant by $\left(\frac{\log((n\Delta_n)^\betaH)}{(n\Delta_n)^{\betaH}}\right)^{\gamma(\boldsymbol{s})}$. To conclude, only note that $(n\Delta_n)^{-\beta_H/2}\le(n\Delta_n)^{-\gamma(\boldsymbol{s})\beta_H} $ since $\gamma(\boldsymbol{s})\le 1/2$.\\

\noindent \textbf{Acknowledgements.}
   The authors have been supported by Fondecyt projects 1171335 and 1190801, and Mathamsud 19-MATH-06 and 20-MATH-05.

   \bibliographystyle{apalike}
   \bibliography{bib1}

   \end{document}